\documentclass[reqno]{amsart}

\makeatletter
\@namedef{subjclassname@2020}{\textup{2020} Mathematics Subject Classification}
\makeatother

\usepackage[english]{babel}
\usepackage{amsmath,amssymb,amsthm,caption,enumerate,booktabs}
\usepackage[bookmarks=false]{hyperref}
\usepackage[centering]{geometry}
\usepackage{xcolor}
\usepackage{color}
\renewcommand{\epsilon}{\varepsilon}            

\newtheorem{theorem}{Theorem}[section]   
\newtheorem*{theorem*}{Theorem}          

\newtheorem{proposition}[theorem]{Proposition}

\theoremstyle{definition}

\newtheorem{example}{Example}[section]

\newtheorem{remark}{Remark}[section]
\newtheorem*{acknow}{Acknowledgments}

\numberwithin{equation}{section}

\title[Homogeneous hypersurfaces of $\mathrm{Sol}_1^4$, $\mathrm{Sol}_{m,n}^4$ and $\mathrm{Nil}^4$]
{Homogeneous hypersurfaces of the four-dimensional Thurston geometries $\mathrm{Sol}_1^4$, $\mathrm{Sol}_{m,n}^4$ and $\mathrm{Nil}^4$}

\author[Xiaoge Lu, Zeke Yao and Xi Zhang]
{Xiaoge Lu, Zeke Yao and Xi Zhang}

\address{School of Mathematics and Statistics, Zhengzhou University,
Zhengzhou 450001, People's Republic of China}
\email{lxgzzu@163.com}

\address{School of Mathematical Sciences, South China Normal University,
Guangzhou 510631, People's Republic of China}	
\email{yaozkleon@163.com,\  zhangxisq@163.com}



\thanks{X. Zhang is the corresponding author}

\begin{document}
	
\begin{abstract}
In this paper, we focus on the four-dimensional Thurston geometries whose isometry groups are four-dimensional, namely $\mathrm{Sol}_1^4$, $\mathrm{Sol}_{m,n}^4$ and $\mathrm{Nil}^4$.
We classify homogeneous hypersurfaces in the above three manifolds.
\end{abstract}
	
\maketitle
	
\section{Introduction}\label{sect:1}

Homogeneous hypersurfaces serve as the fundamental models in isoparametric theory, providing the canonical examples for its complete classification. Here, a hypersurface $M$ in Riemannian manifold $\widetilde{M}$ is called {\it homogeneous} if there exists a closed subgroup $G\subset\operatorname{Iso}_o(\widetilde{M})$ such that $M = G \cdot p=\{g \cdot p \mid g \in G \}$ for some point $p \in \widetilde{M}$, where $\operatorname{Iso}_o(\widetilde{M})$ is the connected component of the identity of isometry group of $\widetilde{M}$.

The classification of homogeneous hypersurfaces is a classical topic. Moreover,
classifying such hypersurfaces is equivalent to classifying cohomogeneity one actions
up to orbit equivalence.
For real space forms, the classification of homogeneous hypersurfaces can be found in references
\cite{CE,HL1,SB,TT}. For nonflat complex space forms, the homogeneous hypersurfaces therein have been classified by Takagi \cite{TR} and Berndt-Tamaru \cite{BT}.
Kollross \cite{KA} classified cohomogeneity one actions on the irreducible symmetric spaces of compact type up to orbit equivalence. Later,
D\'iaz-Ramos, Dom\'inguez-V\'azquez and Otero \cite{DDO1} developed a structural result for cohomogeneity one actions on (not necessarily irreducible) symmetric spaces of noncompact type and arbitrary rank. Most recently, Sanmart\'in-L\'opez and Solonenko \cite{SS} completed the classification of isometric cohomogeneity one actions on such spaces up to orbit equivalence.

For Riemannian products of real space forms, which belong to reducible symmetric spaces, the study of homogeneous hypersurfaces and isoparametric hypersurfaces has achieved many interesting results.
Urbano \cite{UF}, Gao-Ma-Yao \cite{GMY}, Dom\'inguez-V\'azquez and Manzano \cite{DM}, and de Lima and Pipoli \cite{DP} classified homogeneous hypersurfaces and isoparametric hypersurfaces in $\mathbb{S}^2 \times \mathbb{S}^2$, $\mathbb{H}^2 \times \mathbb{H}^2$, $\mathbb{S}^n \times \mathbb{R}$ and $\mathbb{H}^n \times \mathbb{R}$ ($n \geq 2$), respectively.
Gao-Ma-Yao \cite{GMY1} classified isoparametric hypersurfaces in the product space $M_{\kappa_1}^2 \times M_{\kappa_2}^2$ of two-dimensional space forms for $\kappa_1,\kappa_2 \in \{-1,0,1\}$ with $\kappa_1 \neq \kappa_2$. Tan-Xie-Yan \cite{TXY1} established a complete classification of homogeneous hypersurfaces and isoparametric hypersurfaces in $\mathbb{S}^n\times\mathbb{R}^m$ and $\mathbb{H}^n\times\mathbb{R}^m$. Subsequently,
de Lima and Pipoli \cite{DP2} classified homogeneous hypersurfaces and isoparametric hypersurfaces in
$M_{\kappa_1}^{n_1} \times M_{\kappa_2}^{n_2}$ for $n_1,n_2\geq2$ and $\kappa_1,\kappa_2 \in \{-1,0,1\}$ with $|\kappa_1|+|\kappa_2|\neq0$ satisfying a one-point condition.

As far as the authors know, there has been little research on homogeneous hypersurfaces in homogeneous Riemannian manifolds that are non-symmetric spaces.
Notice that $\widetilde{{\rm SL}}(2,\mathbb{R})$,
	${\rm Nil}^3$, ${\rm Sol}^3$ and ${\rm Sol}_0^4$ are homogeneous Riemannian manifolds but not symmetric spaces. The classification of homogeneous surfaces in $\widetilde{{\rm SL}}(2,\mathbb{R})$,
	${\rm Nil}^3$ and ${\rm Sol}^3$ can be found in references \cite{DFO,DM,MP}.
Very recently, D'haene-Wei-Yao-Zhang \cite{DWYZ} established a complete classification for homogeneous hypersurfaces of the four-dimensional Thurston geometry $\mathrm{Sol}_{0}^{4}$. For more developments on the study of homogeneous hypersurfaces and isoparametric theory, we refer the readers to the review articles \cite{BS3,CR,CQ,DDO,GQTY} and the references therein.

Based on the classification of three-dimensional model geometries (known as Thurston geometry), Thurston \cite{TW} proposed the well-known geometrization conjecture: Every compact three-manifold admits a canonical decomposition into pieces, each modeled on one of eight geometries: $\mathbb{R}^3$, $\mathbb{S}^3$, $\mathbb{H}^3$, $\mathbb{S}^2\times\mathbb{R}$, $\mathbb{H}^2\times\mathbb{R}$, $\widetilde{{\rm SL}}(2,\mathbb{R})$,
${\rm Nil}^3$ and ${\rm Sol}^3$. While there is no analog of the geometrization conjecture for four-dimensional manifolds, there are 19 kinds of four-dimensional Thurston geometry (cf. \cite{FR,WCTC}):

\begin{table}[h]
\centering
\renewcommand{\arraystretch}{1.05}
\begin{tabular}{|c|c|}
\hline
Four-dimensional Thurston geometry    & Dimension of isometry group \\  \hline
$\mathbb{R}^4, \, \mathbb{S}^4, \, \mathbb{H}^4$ & 10 \\ \hline
$\mathbb{C}P^2, \, \mathbb{C}H^2$ & 9 \\ \hline
$\mathbb{S}^3 \times \mathbb{R}, \, \mathbb{H}^3 \times \mathbb{R}$ & 7 \\ \hline
$\mathbb{S}^2 \times \mathbb{S}^2, \, \mathbb{H}^2 \times \mathbb{H}^2, \,
\mathbb{S}^2 \times \mathbb{R}^2, \, \mathbb{H}^2 \times \mathbb{R}^2, \,
\mathbb{S}^2 \times \mathbb{H}^2$ & 6 \\ \hline
$\text{Sol}_0^4, \, {\rm F^4}, \, \widetilde{\text{SL}}(2, \mathbb{R}) \times \mathbb{R}, \, \text{Nil}^3 \times \mathbb{R}$ & 5 \\ \hline
$\text{Sol}_1^4, \, \text{Sol}_{m,n}^4, \,\text{Nil}^4$ & 4 \\ \hline
\end{tabular}
\end{table}

In this paper, we will focus on the study of homogeneous hypersurfaces in the four-dimensional Thurston geometry with four-dimensional isometry group. These are the manifolds $\text{Sol}_1^4$, $\text{Sol}_{m,n}^4$ and $\text{Nil}^4$, and they are all homogeneous manifolds and non-symmetric spaces. As main results, we have the following three classification theorems.

\begin{theorem}\label{thm:1.2}
Let $M$ be a homogeneous hypersurface of $\mathrm{Sol}_1^4$. Then up to isometries of $\mathrm{Sol}_1^4$, one of the following two cases occurs:
\begin{enumerate}[\rm(1)]
  \item $M$ is $M_{1,r}$ for some $r\geq0$, which is an orbit through
  $(\tanh r,0,0,-\ln(\cosh r))$ of the subgroup
 $\{(0,x_1,x_2,x_3)\in \mathrm{Sol}_1^4\mid  x_1,x_2,x_3\in \mathbb R\}$, see Example \ref{exa:3.1};
  \item  $M$ is $M_{2,0}$, which is an orbit through the origin
  $(0,0,0,0)$ of the subgroup
 $\{(x_1,x_2,x_3,0)\in \mathrm{Sol}_1^4\mid  x_1,x_2,x_3\in \mathbb R\}$, see Example \ref{exa:3.2}.
\end{enumerate}
\end{theorem}

Before stating our next main result, we assume that $m$ and $n$ are positive integers for which the equation $X^3-m X^2+n X-1=0$ admits three distinct real roots. These roots can be written as $e^{\alpha}$, $e^{\beta}$ and $e^{\gamma}$ with $\alpha<\beta<\gamma$ and $\alpha+\beta+\gamma=0$.

\begin{theorem}\label{thm:1.4}
Let $M$ be a homogeneous hypersurface of $\mathrm{Sol}_{m,n}^4$. Then up to isometries of $\mathrm{Sol}_{m,n}^4$, one of the following five cases occurs:
\begin{enumerate}[\rm(1)]
 \item $M$ is $M_{3,r}$ for some $r\geq0$, which is an orbit through
  $(0,\tfrac{1}{\beta}\tanh(|\beta| r),0,-\tfrac{1}{\beta}\ln(\cosh(|\beta| r)))$ of the subgroup
 $\{(x_1,0,x_2,x_3)\in \mathrm{Sol}_{m,n(m\neq n)}^4\mid  x_{1},x_{2},x_{3}\in \mathbb R\}$, see Example \ref{exa:3.4};
\item $M$ is $M_{4,d}$ for some $0\leq d<1$, which is an orbit through the origin
  $(0,0,0,0)$ of the subgroup
 $\{(x_1,\tfrac{d}{\sqrt{1-d^2}}x_2,x_3,x_2)\in \mathrm{Sol}_{m,m}^4\mid  x_{1},x_{2},x_{3}\in \mathbb R\}$, see Example \ref{exa:3.8};
\item $M$ is $M_{5,r}$ for some $r\geq0$, which is an orbit through
  $(\tfrac{1}{\alpha}\tanh(-\alpha r),0,0,-\tfrac{1}{\alpha}\ln(\cosh(-\alpha r)))$ of the subgroup
 $\{(0,x_1,x_2,x_3)\in \mathrm{Sol}_{m,n}^4\mid  x_{1},x_{2},x_{3}\in \mathbb R\}$, see Example \ref{exa:3.3};
 \item $M$ is $M_{6,r}$ for some $r\geq0$, which is an orbit through
  $(0,0, \tfrac{1}{\gamma}\tanh(\gamma r),-\tfrac{1}{\gamma}\ln(\cosh(\gamma r)))$ of the subgroup
 $\{(x_1,x_2,0,x_3)\in \mathrm{Sol}_{m,n}^4\mid  x_1,x_2,x_3\in \mathbb R\}$, see Example \ref{exa:3.5};
\item $M$ is $M_{7,0}$,  which is an orbit through the origin
  $(0,0,0,0)$ of the subgroup
 $\{(x_1,x_2,x_3,0)\in \mathrm{Sol}_{m,n}^4\mid x_1,x_2,x_3\in \mathbb R\}$,
 see Example \ref{exa:3.6}.
\end{enumerate}
\end{theorem}

\begin{theorem}\label{thm:1.6}
Let $M$ be a homogeneous hypersurface of $\mathrm{Nil}^4$. Then up to isometries of $\mathrm{Nil}^4$, one of the following two cases occurs:
\begin{enumerate}[\rm(1)]
\item $M$ is $M_{8,d}$ for some $0\leq d<1$, which is an orbit through the origin
  $(0,0,0,0)$ of the subgroup
 $\{(x_1,x_2,\tfrac{d}{\sqrt{1-d^2}}x_3,x_3)\in \mathrm{Nil}^4\mid  x_{1},x_{2},x_{3}\in \mathbb R\}$, see Example \ref{exa:3.11};
\item $M$ is $M_{9,0}$, which is an orbit through the origin
  $(0,0,0,0)$ of the subgroup
 $\{(x_1,x_2,x_3,0)\in \mathrm{Nil}^4\mid  x_1,x_2,x_3\in \mathbb R\}$, see Example \ref{exa:3.12}.
\end{enumerate}
\end{theorem}

\begin{remark}\label{rm:1.3}
The isometry group of each of the ambient spaces $\mathrm{Sol}_1^4$, $\mathrm{Sol}_{m,n}^4$ and $\mathrm{Nil}^4$ is four-dimensional and explicitly known. Every homogeneous hypersurface in these spaces can be regarded as the orbit of some three-dimensional subgroup of the isometry group. Classifying homogeneous hypersurfaces is equivalent to classifying the three-dimensional subgroups of the isometry group. Therefore, in this paper we only need to consider the three left invariant metrics defined by the equations \eqref{zx6.1}, \eqref{zx6.2} and \eqref{zx6.3} such that the corresponding isometry groups of $\mathrm{Sol}_1^4$, $\mathrm{Sol}_{m,n}^4$ and $\mathrm{Nil}^4$ are four-dimensional, respectively.
\end{remark}

\begin{remark}\label{rm:1.1}
A hypersurface is called {\it austere} if its multiset of principal curvatures is invariant under change of sign.  This notion was introduced by Harvey-Lawson \cite{HL}  for the construction of special Lagrangian submanifolds in $\mathbb C^n$.  Consequently, every austere hypersurface is automatically minimal. Regarding the metrics considered in this paper defined by \eqref{zx6.1}, \eqref{zx6.2} and \eqref{zx6.3}, all hypersurfaces in the family of examples characterized here are minimal. Among these, the following hypersurfaces are austere: $M_{1,0}$, $M_{2,0}$, $M_{3,0}$, $\{M_{4,d}, 0\leq d<1\}$, $M_{5,0}$ of $\mathrm{Sol}_{m,n}^4$, $\{M_{5,r}, r>0\}$ of $\mathrm{Sol}_{m,m}^4$, $M_{6,0}$ of $\mathrm{Sol}_{m,n}^4$, $\{M_{6,r}, r>0\}$ of $\mathrm{Sol}_{m,m}^4$, $M_{7,0}$ of $\mathrm{Sol}_{m,m}^4$, $\{M_{8,d}, 0\leq d<1\}$ and $M_{9,0}$.
\end{remark}

\begin{remark}\label{rm:1.2}
All homogeneous hypersurfaces characterized here have no focal manifolds.
For $i\in \{2,7,9\}$, all parallel hypersurfaces of $M_{i,0}$ are congruent to $M_{i,0}$.
In contrast, for $j\in \{1,3,5,6\}$, any two distinct hypersurfaces in $\{M_{j,r}, r\geq0\}$
are non-congruent.
For any fixed $d$, the parallel hypersurfaces of $M_{4,d}$ (resp. $M_{8,d}$) are congruent to
$M_{4,d}$ (resp. $M_{8,d}$).
However, for $d_1\neq d_2$, $M_{4,d_1}$ and $M_{4,d_2}$ (resp. $M_{8,d_1}$ and $M_{8,d_2}$) are non-congruent.
Besides, the hypersurfaces $M_{7,0}$ and $M_{9,0}$ are flat.
\end{remark}

\begin{remark}[Added on June 29, 2026]\label{rm:1.6}
After our manuscript had been submitted for publication on June 3, 2026, we became aware of a recent arXiv preprint, ``Homogeneous Hypersurfaces in 4-dimensional Thurston Geometries with 4-dimensional Isometry Group", by Tarcios Andrey Ferreira (arXiv:2606.28206). Ferreira's work classifies, up to conjugacy, the 3-dimensional subalgebras of the Lie algebras associated with the 4-dimensional Thurston geometries whose isometry groups have dimension 4. Our work is self-contained and was carried out entirely independently. Although the two works share the same topic, our approach is substantially different from that of Ferreira. We post our paper so that our independent contribution is available to the community.
\end{remark}

This paper is organized as follows: In Sect. 2, we review and
collect the basic materials of ambient spaces $\mathrm{Sol}_1^4$, $\mathrm{Sol}_{m,n}^4$, $\mathrm{Nil}^4$ and their hypersurfaces. In Sect. 3, we describe in some details the examples of homogeneous hypersurfaces, which appear in Theorems \ref{thm:1.2}--\ref{thm:1.6}. Finally, in each of the subsequent three sections, we prove Theorems \ref{thm:1.2}, \ref{thm:1.4} and \ref{thm:1.6}, respectively.

\begin{acknow}
		Z. Yao was supported by National Natural Science Foundation of China (Grant No. 12401061).
		X. Zhang was supported by the China Postdoctoral Science Foundation (Grant No. 2025M773115).
	\end{acknow}

\section{Preliminaries}\label{sect:2}
For convenience, we use the same symbol $g$ to denote the Riemannian metric on each of the three ambient spaces $\mathrm{Sol}_1^4$, $\mathrm{Sol}_{m,n}^4$ and $\mathrm{Nil}^4$.
The specific one should be seen clearly from context.

\subsection{The geometric structure on $\mathrm{Sol}_1^4$}\label{sect:2.1}~

In this subsection, we review some basic materials about $\mathrm{Sol}_1^4$ from \cite{EI,EI1}.
The underlying manifold of the model space $\mathrm{Sol}_1^4$ is the following solvable Lie group
\begin{equation*}
\left\{ (x, y, z, t) :=
\begin{pmatrix}
1 & 0 & e^{-t}x & z \\
0 & e^t & 0 & x \\
0 & 0 & e^{-t} & y \\
0 & 0 & 0 & 1
\end{pmatrix} : x, y, z, t \in \mathbb{R} \right\}.
\end{equation*}
The group multiplication is given by
\begin{equation}\label{kkk:2.0}
(x_1, y_1, z_1, t_1)(x_2, y_2, z_2, t_2)=(x_1 + e^{t_1}x_2, y_1 + e^{-t_1}y_2,
z_1 + z_2 + e^{-t_1}x_1y_2, t_1 + t_2).
\end{equation}

At a point $p = (x, y, z, t) \in \mathrm{Sol}_1^4$, there are four left invariant vector fields:
\begin{equation}\label{kkk:2.1}
E_1=e^t\partial_x,\ \ E_2=e^{-t}(\partial_y+x\partial_z),\ \ E_3=\partial_z,\ \ E_4=\partial_t.
\end{equation}
Then one can check the following brackets:
\begin{equation}\label{kkk:2.2}
\begin{aligned}
&[E_1, E_2]=E_3, &  [E_1, E_3]&=0, & [E_1, E_4]&=-E_1, \\
&[E_2, E_3]=0, & [E_2, E_4]&=E_2, & [E_3, E_4]&=0.
\end{aligned}
\end{equation}

The group that acts on $\mathrm{Sol}_1^4$ making it a Thurston geometry is $\mathrm{Sol}_1^4\rtimes D_4$ (see \cite{EI1}). Here, $\mathrm{Sol}_1^4$ acts on itself by left translations, and $D_4$ denotes the dihedral group of 8 elements and is generated by the maps
$\phi_i: \mathrm{Sol}_1^4\rightarrow\mathrm{Sol}_1^4$ for $1\leq i\leq4$, defined by
\begin{equation}\label{fff:2.4}
\begin{aligned}
&\phi_1(x,y,z,t)=(-x,y,-z,t), \ \ \phi_2(x,y,z,t)=(-x,-y,z,t),\\
&\phi_3(x,y,z,t)=(x,-y,-z,t),\ \ \phi_4(x,y,z,t)=(y,x,-z+xy,-t).
\end{aligned}
\end{equation}

In this paper, we select the Riemannian metric $g$ on $\mathrm{Sol}_1^4$ as defined in reference \cite{EI}:
\begin{equation}\label{zx6.1}
g=e^{-2t}dx^2 + e^{2t}dy^2 + (dz-x dy)^2 + dt^2.
\end{equation}
Then the frame field $\{E_i\}^4_{i=1}$ is orthonormal with respect to $g$.

Let $\tilde{\nabla}$ be the Levi-Civita connection of $g$. Applying the equation \eqref{kkk:2.2} and the Koszul's formula, we derive
\begin{equation}\label{kkk:2.3}
\begin{aligned}
\tilde{\nabla}_{E_1}E_1 &= E_4,      & \tilde{\nabla}_{E_1}E_2 &= \tfrac{1}{2} E_3,       & \tilde{\nabla}_{E_1}E_3 &= -\tfrac{1}{2} E_2,       & \tilde{\nabla}_{E_1}E_4 &= -E_1, \\
\tilde{\nabla}_{E_2}E_1 &= -\tfrac{1}{2}E_3,       & \tilde{\nabla}_{E_2}E_2 &=-E_4,     & \tilde{\nabla}_{E_2}E_3 &= \tfrac{1}{2}E_1,       & \tilde{\nabla}_{E_2}E_4 &= E_2, \\
\tilde{\nabla}_{E_3}E_1 &= -\tfrac{1}{2}E_2,       & \tilde{\nabla}_{E_3}E_2 &=\tfrac{1}{2}E_1,       & \tilde{\nabla}_{E_3}E_3 &=0,   & \tilde{\nabla}_{E_3}E_4 &=0,  \\
\tilde{\nabla}_{E_4}E_1 &= 0,       & \tilde{\nabla}_{E_4}E_2 &= 0,       & \tilde{\nabla}_{E_4}E_3 &= 0,       & \tilde{\nabla}_{E_4}E_4 &= 0.
\end{aligned}
\end{equation}

\subsection{The geometric structure on $\mathrm{Sol}_{m,n}^4$}\label{sect:2.2}~

In this subsection, we collect some necessary materials about $\mathrm{Sol}_{m,n}^4$ from \cite{BM,DM1}.
The underlying manifold of $\mathrm{Sol}_{m,n}^4$ is a solvable Lie group consisting of the matrices
$$
\begin{pmatrix}
e^{\alpha t} & 0 & 0& x \\
0 & e^{\beta t} & 0 & y \\
0 & 0 & e^{\gamma t} &z \\
0 & 0 & 0 & 1
\end{pmatrix},
$$
where $x,y,z,t,\alpha,\beta,\gamma\in \mathbb R$ and $\alpha<\beta<\gamma$ satisfying $\alpha+\beta+\gamma=0$. The values $e^{\alpha}$, $e^{\beta}$ and $e^{\gamma}$ are the roots of the equation $X^3-m X^2+n X-1=0$. Here,  $m$ and $n$ are positive integers such that the roots are all real and distinct.
The group operation is given explicitly by
\begin{equation}\label{ttt:2.0}
(x_1, y_1, z_1, t_1)(x_2, y_2, z_2, t_2) = (x_1 + e^{\alpha t_1}x_2, y_1 + e^{\beta t_1}y_2, z_1 + e^{\gamma t_1}z_2 , t_1 + t_2).
\end{equation}

At a point $p = (x, y, z, t) \in \mathrm{Sol}_{m,n}^4$, the left invariant vector fields are given by
\begin{align}\label{ttt:2.1}
E_1=e^{\alpha t}\partial_x,\ \ E_2=e^{\beta t}\partial_y,\ \ E_3=e^{\gamma t}\partial_z,\ \ E_4=\partial_t.
\end{align}
Then we have the following commutation relations:
\begin{equation}\label{ttt:2.2}
\begin{aligned}
&[E_1, E_2]=0, &  [E_1, E_3]&=0, & [E_1, E_4]&=-\alpha E_1, \\
&[E_2, E_3]=0, & [E_2, E_4]&=- \beta E_2, & [E_3, E_4]&=-\gamma E_3.
\end{aligned}
\end{equation}

When $m\neq n$, the group associated to this Thurston geometry is $\mathrm{Sol}_{m,n(m\neq n)}^4 \rtimes (\mathbb{Z}/2\mathbb{Z})^3$, where $\mathrm{Sol}_{m,n(m\neq n)}^4$ acts on itself by left translations, and each $\mathbb{Z}/2\mathbb{Z}$ factor is generated by reflecting either the $x$-, $y$- or $z$-coordinate (see \cite{DM1}). When $m=n$, we have $\beta=0$, $\alpha=-\gamma$, and $\mathrm{Sol}_{m,m}^4$
can be identified with $\mathrm{Sol}^3 \times \mathbb R$.
The associated group is $\mathrm{Sol}_{m,m}^4 \rtimes (D_4 \times (\mathbb{Z}/2\mathbb{Z}))$, where $\mathbb{Z}/2\mathbb{Z}$ is generated by reflecting the
$y$-coordinate, and $D_4$ is generated by the maps
$\varphi_i:\mathrm{Sol}_{m,m}^4\rightarrow\mathrm{Sol}_{m,m}^4$ for $1\leq i\leq3$, defined by
$$
\varphi_1(x,y,z,t)=(-x,y,z,t), \ \ \varphi_2(x,y,z,t)=(x,y,-z,t), \ \
\varphi_3(x,y,z,t)=(z,y,x,-t).
$$

We choose the Riemannian metric $g$ on $\mathrm{Sol}_{m,n}^4$ as defined in reference \cite{BM}:
\begin{equation}\label{zx6.2}
g=e^{-2\alpha t}dx^2 + e^{-2\beta t}dy^2 + e^{-2\gamma t} dz^2 + dt^2.
\end{equation}
Then the frame field $\{E_i\}^4_{i=1}$ is orthonormal with respect to $g$.

Let $\tilde{\nabla}$ be the Levi-Civita connection of $g$. By using \eqref{ttt:2.2} and Koszul's formula, we obtain
\begin{equation}\label{ttt:2.3}
\begin{aligned}
\tilde{\nabla}_{E_1}E_1 &=\alpha E_4,      & \tilde{\nabla}_{E_1}E_2 &= 0,
& \tilde{\nabla}_{E_1}E_3 &=0,       & \tilde{\nabla}_{E_1}E_4 &= -\alpha E_1, \\
\tilde{\nabla}_{E_2}E_1 &=0,       & \tilde{\nabla}_{E_2}E_2 &=\beta E_4,
& \tilde{\nabla}_{E_2}E_3 &=0,       & \tilde{\nabla}_{E_2}E_4 &=-\beta E_2, \\
\tilde{\nabla}_{E_3}E_1 &=0,       & \tilde{\nabla}_{E_3}E_2 &=0,
& \tilde{\nabla}_{E_3}E_3 &=\gamma E_4,   & \tilde{\nabla}_{E_3}E_4 &=-\gamma E_3,  \\
\tilde{\nabla}_{E_4}E_1 &= 0,       & \tilde{\nabla}_{E_4}E_2 &= 0,       & \tilde{\nabla}_{E_4}E_3 &= 0,       & \tilde{\nabla}_{E_4}E_4 &= 0.
\end{aligned}
\end{equation}

\subsection{The geometric structure on $\mathrm{Nil}^4$}\label{sect:2.3}~

In this subsection, we review some basic materials about $\mathrm{Nil}^4$ from \cite{DM1,DHCB}.
The underlying manifold of $\mathrm{Nil}^4$ is the following solvable Lie group:
$$
\left\{ (x, y, z, t) :=
\begin{pmatrix}
1 & t & \frac{t^2}{2}& x \\
0 & 1 & t & y \\
0 & 0 &1 &z \\
0 & 0 & 0 & 1
\end{pmatrix} : x, y, z, t \in \mathbb{R} \right\}.
$$
The group multiplication is
\begin{equation}\label{hhh:2.0}
(x_1, y_1, z_1, t_1)(x_2, y_2, z_2, t_2)
=(x_1+x_2+t_1y_2+\tfrac{t_1^2}{2}z_2, y_1+y_2+t_1z_2, z_1+z_2, t_1+t_2).
\end{equation}

At a point $p = (x, y, z, t) \in \mathrm{Nil}^4$, the left invariant vector fields are given by
\begin{align}\label{hhh:2.1}
E_1=\partial_x,\ \ E_2=t\partial_x+\partial_y,\ \ E_3=\tfrac{t^2}{2}\partial_x+t \partial_y +\partial_z,\ \ E_4=\partial_t.
\end{align}
Then we obtain the following equations:
\begin{equation}\label{hhh:2.2}
\begin{aligned}
&[E_1, E_2]=0, &  [E_1, E_3]&=0, & [E_1, E_4]&=0, \\
&[E_2, E_3]=0, &  [E_2, E_4]&=- E_1, &
[E_3, E_4]&=- E_2.
\end{aligned}
\end{equation}

According to \cite{DM1}, the group associated to $\mathrm{\mathrm{Nil}}^4$ is $\mathrm{Nil}^4 \rtimes (\mathbb{Z}/2\mathbb{Z})^2$, where $\mathrm{Nil}^4$ acts on itself by left translations, and the copies of $\mathbb{Z}/2\mathbb{Z}$ is generated by the maps $\psi_i:\mathrm{Nil}^4\rightarrow\mathrm{Nil}^4$ for $1\leq i\leq2$, defined by
\begin{equation}\label{zx:1.4}
\psi_1(x,y,z,t)=(-x,-y,-z,t), \ \ \psi_2=(x,y,z,t)=(x,-y,z,-t).
\end{equation}

We select the Riemannian metric $g$ on $\mathrm{Nil}^4$ as defined in reference \cite{DHCB}:
\begin{equation}\label{zx6.3}
g=dx^2 -2t\, dx dy+t^2\, dx dz+(1+t^2)dy^2-2t(1+\tfrac{t^2}{2})dy dz+(1+t^2+\tfrac{t^4}{4})dz^2 + dt^2.
\end{equation}
Then the frame field $\{E_i\}^4_{i=1}$ is orthonormal with respect to $g$.

Let $\tilde{\nabla}$ be the Levi-Civita connection of $g$. By using \eqref{hhh:2.2} and Koszul's formula, we obtain
\begin{equation}\label{hhh:2.3}
\begin{aligned}
\tilde{\nabla}_{E_1}E_1 &=0,      & \tilde{\nabla}_{E_1}E_2 &=\tfrac{1}{2}E_4,
& \tilde{\nabla}_{E_1}E_3 &=0,       & \tilde{\nabla}_{E_1}E_4 &=-\tfrac{1}{2}E_2, \\
\tilde{\nabla}_{E_2}E_1 &=\tfrac{1}{2}E_4,       & \tilde{\nabla}_{E_2}E_2 &=0,
& \tilde{\nabla}_{E_2}E_3 &=\tfrac{1}{2}E_4,       & \tilde{\nabla}_{E_2}E_4 &=
-\tfrac{1}{2}E_1-\tfrac{1}{2}E_3, \\
\tilde{\nabla}_{E_3}E_1 &=0,       & \tilde{\nabla}_{E_3}E_2 &=\tfrac{1}{2}E_4,
& \tilde{\nabla}_{E_3}E_3 &=0,   & \tilde{\nabla}_{E_3}E_4 &=-\tfrac{1}{2}E_2,  \\
\tilde{\nabla}_{E_4}E_1 &= -\tfrac{1}{2}E_2,       & \tilde{\nabla}_{E_4}E_2 &=\tfrac{1}{2}E_1-\tfrac{1}{2}E_3,       & \tilde{\nabla}_{E_4}E_3 &=\tfrac{1}{2}E_2,       & \tilde{\nabla}_{E_4}E_4 &= 0.
\end{aligned}
\end{equation}

\subsection{Hypersurfaces of $\mathrm{Sol}_1^4$, $\mathrm{Sol}_{m,n}^4$ and $\mathrm{Nil}^4$}\label{sect:2.4}~

In this subsection, we uniformly describe the basic theory of hypersurfaces in $\mathrm{Sol}_1^4$, $\mathrm{Sol}_{m,n}^4$ and $\mathrm{Nil}^4$.
Let $M$ be an isometrically immersed hypersurface of any of $\mathrm{Sol}_1^4$, $\mathrm{Sol}_{m,n}^4$ or $\mathrm{Nil}^4$ with unit normal vector field $N$. We can assume that
$$
N = aE_1 + bE_2 + cE_3 + dE_4,
$$
where $\{E_i\}_{i=1}^4$ is defined by \eqref{kkk:2.1}
(resp. \eqref{ttt:2.1}, resp. \eqref{hhh:2.1}), and $a,b,c,d$ are smooth functions on $M$,
called {\it angle functions} satisfying $a^2 + b^2 + c^2 + d^2 = 1$. Then the following tangent vector fields
are orthonormal on $M$:
\begin{equation}\label{kkk:2.5}
\begin{aligned}
T_1 = bE_1 - aE_2 + dE_3 - cE_4, \\
T_2 = cE_1 - dE_2 - aE_3 + bE_4, \\
T_3 = dE_1 + cE_2 - bE_3 - aE_4.
\end{aligned}
\end{equation}

Let $\nabla$ be the Levi-Civita connection of the induced metric $g$
on $M$. Then the Gauss and Weingarten formulae are given respectively as below:
\begin{equation}\label{kkk:2.6}
\tilde{\nabla}_XY=\nabla_XY+g(AX,Y) N,\ \
\tilde{\nabla}_XN=-AX,\ \hbox{for any}\  X,Y\in TM,
\end{equation}
where $A$ is the shape operator of $M$.

Let $\tilde{R}$ and $R$ be the Riemannian curvature tensors of the ambient space and $M$, respectively. Then, the Gauss and Codazzi equations are given by:
\begin{equation}\label{kkk:2.7}
\begin{aligned}
R(X,Y)Z&=[\tilde{R}(X,Y)Z]^\top+g(AY,Z)AX-g(AX,Z) AY,
\end{aligned}
\end{equation}
\begin{equation}\label{kkk:2.8}
(\nabla_XA)Y-(\nabla_YA)X=-[\tilde{R}(X,Y)N]^\top,
\end{equation}
where $X,Y,Z\in TM$ and $\cdot^\top$ denotes the tangential component.

\section{Examples of homogeneous hypersurfaces}\label{sect:3}
In this section, we construct examples of homogeneous hypersurfaces in $\mathrm{Sol}_1^4$, $\mathrm{Sol}_{m,n}^4$ and $\mathrm{Nil}^4$.

\subsection{Examples of homogeneous hypersurfaces of $\mathrm{Sol}_1^4$}\label{sect:3.1}~
\begin{example}\label{exa:3.1}
For any given $r\geq0$, we define the hypersurface
$$
M_{1,r}:=\{(e^{x_3}\tanh r,x_1,x_2,x_3-\ln(\cosh r))\in \mathrm{Sol}_1^4\mid x_{1},x_{2},x_{3}\in\mathbb{R}\}.
$$
Put $H_1:=\{(0,x_1,x_2,x_3)\in \mathrm{Sol}_1^4\mid  x_1,x_2,x_3\in\mathbb R\}\subset\operatorname{Iso}_o(\mathrm{Sol}_1^4)=\mathrm{Sol}_1^4$.
\end{example}
\begin{proposition}\label{prop:3.1}
The hypersurface $M_{1,r}$ has the following properties:
\begin{enumerate}[{\rm(1)}]
\item The hypersurface $M_{1,r}$ can also be presented as $\{(x,y,z,t)\in \mathrm{Sol}_1^4\mid x e^{-t}=\sinh r \}$. The unit normal vector field is given by $N=-\operatorname{sech} r\, E_1 + \tanh r\, E_4$, where $\{E_i\}^4_{i=1}$ is defined by \eqref{kkk:2.1};
\item It is a minimal hypersurface and has constant principal curvatures $\tfrac{-1-\tanh r}{2}$, $\tfrac{1-\tanh r}{2}$ and $\tanh r$. When $r=\tfrac{1}{2} \ln 2$, the hypersurface $M_{1,r}$ has two distinct
    constant principal curvatures $-\tfrac{2}{3}$ and $\tfrac{1}{3}$ (multiplicity 2). When $r\in[0,\frac{1}{2} \ln2)\cup(\frac{1}{2} \ln 2,+\infty)$, $M_{1,r}$ has three distinct
    constant principal curvatures.
    The eigenvalues of the Ricci tensor of $M_{1,r}$ are $\tfrac{\tanh^2 r}{2}$ and $-1+\tfrac{\tanh^2 r}{2}$ (multiplicity 2);
\item It is an orbit of the subgroup $H_1$ which passes through the point $(\tanh r,0,0,-\ln(\cosh r))$. Therefore, $M_{1,r}$ is a homogeneous hypersurface, and it has no focal manifold.
\end{enumerate}
\end{proposition}
\begin{proof}
(1) At any point  $p\in M_{1,r}$, we consider the frame field $\{V_i\}^3_{i=1}$ on $M_{1,r}$:
$$
V_1=p_{x_1}=e^{x_3}({\operatorname{sech} r\, E_2-\tanh r\, E_3}), \ \
V_2=p_{x_2}=E_3, \ \
V_3=p_{x_3}=\sinh r\, E_1+E_4.
$$
Thus $N=-\operatorname{sech} r\, E_1 + \tanh r\, E_4$ is a unit normal vector field of $M_{1,r}$.

(2) Choose the orthonormal frame field  $\{W_i\}^3_{i=1}$ on $M_{1,r}$ as follows:
\begin{equation}\label{lll:3.1}
W_1=\tanh r E_1+\operatorname{sech} r E_4, \ \ W_2=E_2,\ \
W_3=E_3.
\end{equation}
According to \eqref{kkk:2.3}, \eqref{lll:3.1} and the Gauss and Weingarten formulae, we get
\begin{equation}\label{lll:3.111}
\begin{aligned}
\nabla_{W_1}W_1 &=0,
&\nabla_{W_1}W_2 &= \tfrac{\tanh r}{2}W_3,
&\nabla_{W_1}W_3 &=-\tfrac{\tanh r}{2}W_2, \\
\nabla_{W_2}W_1 &=\operatorname{sech} r W_2-\tfrac{\tanh r}{2}W_3,
&\nabla_{W_2}W_2 &=-\operatorname{sech} r W_1,
&\nabla_{W_2}W_3 &=\tfrac{\tanh r}{2}W_1, \\
\nabla_{W_3}W_1 &=-\tfrac{\tanh r}{2}W_2,       & \nabla_{W_3}W_2&=\tfrac{\tanh r}{2}W_1,
& \nabla_{W_3}W_3&=0,
\end{aligned}
\end{equation}
\begin{equation}\label{lll:3.2}
AW_1=\tanh r\, W_1, \ AW_2=-\tanh r\, W_2-\tfrac{\operatorname{sech} r}{2}W_3,\
AW_3=-\tfrac{\operatorname{sech} r}{2} W_2.
\end{equation}
It follows from \eqref{lll:3.2} that the principal curvatures of $M_{1,r}$ are $\tfrac{-1-\tanh r}{2}$, $\tfrac{1-\tanh r}{2}$ and $\tanh r$. We know that $\tfrac{-1-\tanh r}{2}\neq \tfrac{1-\tanh r}{2}$. Since $r\geq0$, we also have $\frac{-1-\tanh r}{2}\neq \tanh r$.
If $\tfrac{1-\tanh r}{2}=\tanh r$, we get $\tanh r=\frac{1}{3}$ and then $r=\tfrac{1}{2} \ln 2$,
which implies $M_{1,r}$ has two distinct
constant principal curvatures $-\tfrac{2}{3}$ and $\tfrac{1}{3}$ (multiplicity 2).

By using \eqref{lll:3.111}, we can calculate the sectional curvature and Ricci curvature of $M_{1,r}$ through their definitions and obtain
$$
K(W_1\wedge W_2)=-1+\tfrac{\tanh^2 r}{4},\ \ K(W_1\wedge W_3)=K(W_2\wedge W_3)=\tfrac{\tanh^2 r}{4},
$$
$$
\mathrm{Ric}(W_1)=\big(-1+\tfrac{\tanh^2 r}{2}\big) W_1,\ \
\mathrm{Ric}(W_2)=\big(-1+\tfrac{\tanh^2 r}{2}\big)W_2,\ \
\mathrm{Ric}(W_3)=\tfrac{\tanh^2 r}{2} W_3.
$$

(3) It follows from \eqref{kkk:2.0} that $H_1$ is a closed subgroup of $\mathrm{Sol}_1^4$. Moreover, by \eqref{kkk:2.0} we further derive that the hypersurface $M_{1,r}$ is an orbit of the subgroup $H_1$ which passes through the point $(\tanh r,0,0,-\ln(\cosh r))$,  and it has no focal manifold.
\end{proof}

\begin{example}[cf. Example 7.1 of \cite{EI}]\label{exa:3.2}
For any given $r \in \mathbb{R}$, we consider the hypersurface $M_{2,r}$ defined by
$$
M_{2,r} := \{(x_1, x_2, x_3, r) \in \mathrm{Sol}_1^4 \mid x_1, x_2, x_3 \in \mathbb{R}\}.
$$
\end{example}
Let $H_2:= \{(x_1, x_2, x_3, 0) \in \mathrm{Sol}_1^4 \mid x_1, x_2, x_3 \in \mathbb{R}\}$. Then by \eqref{kkk:2.0}, we know that $H_2$ is a closed subgroup of $\mathrm{Sol}_1^4$. Thus, by using \eqref{kkk:2.0}, \eqref{kkk:2.1}, \eqref{kkk:2.3}, the
Gauss and Weingarten formulae, and the definition of Ricci curvature, we have the following proposition without proof.
\begin{proposition}\label{prop:3.2}
The hypersurface $M_{2,r}$ has the following properties:
\begin{enumerate}[\rm(1)]
    \item The hypersurface $M_{2,r}$ can also be presented as $\{(x,y,z,t)\in \mathrm{Sol}_1^4\mid t=r\}$. The unit normal vector field is given by $N = E_4$, where $E_4$ is defined by \eqref{kkk:2.1};
    \item It is a minimal hypersurface and has three distinct constant principal curvatures $-1$, $0$ and $1$. The eigenvalues of the Ricci tensor of $M_{2,r}$ are $\tfrac{1}{2}$ and $-\tfrac{1}{2}$ (multiplicity 2);
    \item It is an orbit of subgroup $H_2$ which passes through the point $(0, 0, 0, r)$. Therefore, the hypersurface $M_{2,r}$ is homogeneous, and it has no focal manifold.
        Moreover, for any given $r$, the hypersurface $M_{2,r}$ is congruent to $M_{2,0}$.
\end{enumerate}
\end{proposition}

\subsection{Examples of homogeneous hypersurfaces of $\mathrm{Sol}_{m,n}^4$}\label{sect:3.2}~

First, for $m\neq n$, we can construct a family of examples. Before stating examples,
we recall that $m$ and $n$ are positive integers such that the roots of the equation $X^3-m X^2+n X-1=0$ are all real and distinct. The roots of the above equation are $e^{\alpha}$, $e^{\beta}$ and $e^{\gamma}$, where $\alpha<\beta<\gamma$ and $\alpha+\beta+\gamma=0$.

\begin{example}\label{exa:3.4}
For any given $r\geq0$, we define the hypersurface
$$
M_{3,r}:=\{(x_1,\tfrac{1}{\beta}e^{\beta x_3}\tanh (|\beta| r),x_2,x_3-\tfrac{1}{\beta}\ln(\cosh (|\beta| r)))\in \mathrm{Sol}_{m,n(m\neq n)}^4\mid x_{1},x_{2},x_{3}\in\mathbb{R}\}.
$$
Put $H_3:=\{(x_1,0,x_2,x_3)\in \mathrm{Sol}_{m,n(m\neq n)}^4\mid  x_{1},x_{2},x_{3}\in \mathbb R\}\subset\operatorname{Iso}_o(\mathrm{Sol}_{m,n(m\neq n)}^4)=\mathrm{Sol}_{m,n(m\neq n)}^4$.
\end{example}
\begin{proposition}\label{prop:3.4}
The hypersurface $M_{3,r}$ has the following properties:
\begin{enumerate}[{\rm(1)}]
\item The hypersurface $M_{3,r}$ can also be presented as
$\{(x,y,z,t)\in \mathrm{Sol}_{m,n(m\neq n)}^4\mid
y e^{-\beta t}=\tfrac{\sinh(|\beta| r)}{\beta}\}$.
The unit normal vector field is given by $N=-\operatorname{sech}(|\beta| r)\, E_2 + \tanh(|\beta| r)\, E_4$, where $\{E_i\}^4_{i=1}$ is defined
by \eqref{ttt:2.1};
\item It is a minimal hypersurface and has constant principal curvatures $\alpha \tanh (|\beta| r)$, $\beta \tanh (|\beta| r)$ and $\gamma \tanh (|\beta| r)$. When $r=0$, $M_{3,0}$ is a totally geodesic hypersurface. The eigenvalues of the Ricci tensor of $M_{3,r}$ are $-\alpha(\alpha+\gamma)\operatorname{sech}^2(|\beta| r)$,  $-(\alpha^2+\gamma^2)\operatorname{sech}^2(|\beta| r)$ and $-\gamma(\alpha+\gamma)\operatorname{sech}^2(|\beta| r)$;
\item It is an orbit through $(0,\tfrac{1}{\beta}\tanh(|\beta| r),0,-\tfrac{1}{\beta}\ln(\cosh(|\beta| r)))$ of the subgroup $H_3$. Therefore, $M_{3,r}$ is a homogeneous hypersurface, and it has no focal manifold.
\end{enumerate}
\end{proposition}
\begin{proof}
(1) At any point  $p\in M_{3,r}$, we consider the frame field $\{V_i\}^3_{i=1}$ as follows:
$$
\begin{aligned}
V_1&=p_{x_1}=e^{-\alpha (x_3-\tfrac{1}{\beta}\ln(\cosh (|\beta| r)))}E_1, \ \
V_2=p_{x_2}=e^{-\gamma(x_3-\tfrac{1}{\beta}\ln(\cosh (|\beta| r)))}E_3,\\
V_3&=p_{x_3}=\sinh(|\beta| r)\, E_2+E_4.
\end{aligned}
$$
Then, $N=-\operatorname{sech}(|\beta| r)\, E_2 + \tanh(|\beta| r)\, E_4$ is a unit normal vector field.

(2) We choose the orthonormal frame field  $\{W_i\}^3_{i=1}$ on $M_{3,r}$ as follows:
\begin{equation}\label{iii:3.1}
W_1=E_1,\
W_2=\tfrac{V_3}{\cosh(|\beta| r)}=\tanh(|\beta| r)\, E_2+\operatorname{sech}(|\beta| r)\, E_4,\
W_3=E_3.
\end{equation}
From \eqref{ttt:2.3}, \eqref{iii:3.1}, the Gauss and Weingarten formulae, we get
\begin{equation}\label{iii:3.000}
\begin{aligned}
&\nabla_{W_1}W_1 =\operatorname{sech}(|\beta| r) \alpha W_2,\ \ \ \ \ \ \,
\nabla_{W_1}W_2 =-\operatorname{sech}(|\beta| r) \alpha W_1,\\
&\nabla_{W_3}W_2=-\operatorname{sech}(|\beta| r) \gamma W_3,\ \ \ \
\nabla_{W_3}W_3=\operatorname{sech}(|\beta| r)\gamma W_2,\\
&\nabla_{W_1}W_3=
\nabla_{W_2}W_1=
\nabla_{W_2}W_2=
\nabla_{W_2}W_3=
\nabla_{W_3}W_1 =0,
\end{aligned}
\end{equation}
\begin{equation}\label{iii:3.2}
AW_1=\alpha \tanh(|\beta| r)\, W_1, \ AW_2=\beta\tanh(|\beta| r)W_2,\
AW_3=\gamma\tanh(|\beta| r)\, W_3.
\end{equation}
The equation \eqref{iii:3.2} shows that the principal curvatures of $M_{3,r}$
are $\alpha \tanh (|\beta| r)$, $\beta \tanh (|\beta| r)$ and $\gamma \tanh (|\beta| r)$.

It follows from \eqref{iii:3.000} that the sectional curvature and Ricci curvature of $M_{3,r}$ are
$$
\begin{aligned}
&K(W_1\wedge W_2)=-\alpha^2\operatorname{sech}^2(|\beta| r),\ \
K(W_1\wedge W_3)=-\alpha\gamma\operatorname{sech}^2(|\beta| r),\\
&K(W_2\wedge W_3)=-\gamma^2\operatorname{sech}^2(|\beta| r),\ \
\mathrm{Ric}(W_1)=-\alpha(\alpha+\gamma)\operatorname{sech}^2(|\beta| r) W_1,\\
&\mathrm{Ric}(W_2)=-(\alpha^2+\gamma^2)\operatorname{sech}^2(|\beta| r) W_2,\ \
\mathrm{Ric}(W_3)=-\gamma(\alpha+\gamma)\operatorname{sech}^2(|\beta| r) W_3.
\end{aligned}
$$

(3) According to \eqref{ttt:2.0}, we can directly verify that $H_3$ is a closed subgroup of $\mathrm{Sol}_{m,n(m\neq n)}^4$. Moreover, it follows from \eqref{ttt:2.0} that the hypersurface $M_{3,r}$ is an orbit of the subgroup $H_3$ which passes through $(0,\tfrac{1}{\beta}\tanh(|\beta| r),0,-\tfrac{1}{\beta}\ln(\cosh(|\beta| r)))$,  and it has no focal manifold.
\end{proof}

Next, for $m=n$, we can construct the following family of examples.

\begin{example}\label{exa:3.8}
For any fixed $0\leq d<1$, we define the hypersurface
$$
M_{4,d}:=\{(x_1,\tfrac{d}{\sqrt{1-d^2}}x_2,x_3,x_2)\in \mathrm{Sol}_{m,m}^4\mid x_{1},x_{2},x_{3}\in\mathbb{R}\}.
$$
Put $H_4:=\{(x_1,\tfrac{d}{\sqrt{1-d^2}}x_2,x_3,x_2)\in \mathrm{Sol}_{m,m}^4\mid  x_{1},x_{2},x_{3}\in \mathbb R\}\subset\operatorname{Iso}_o(\mathrm{Sol}_{m,m}^4)=\mathrm{Sol}_{m,m}^4$.
\end{example}
\begin{proposition}\label{prop:3.8}
The hypersurface $M_{4,d}$ has the following properties:
\begin{enumerate}[{\rm(1)}]
\item The hypersurface $M_{4,d}$ can also be presented as $\{(x,y,z,t)\in \mathrm{Sol}_{m,m}^4\mid dt-\sqrt{1-d^2}y=0\}$.
The unit normal vector field is given by $N=-\sqrt{1-d^2}\, E_2 + d\, E_4$, where $\{E_i\}^4_{i=1}$ is defined
by \eqref{ttt:2.1};
\item It is a minimal hypersurface and has constant principal curvatures  $\alpha d$, $0$ and $-\alpha d$. When $d=0$, $M_{4,0}$ is a totally geodesic hypersurface.
The eigenvalues of the Ricci tensor of $M_{4,d}$ are $2\alpha^2(d^2-1)$ and $0$ (multiplicity 2);
\item It is an orbit through $(0,0,0,0)$ of the subgroup $H_4$. Therefore, $M_{4,d}$ is a homogeneous hypersurface, and it has no focal manifold. Moreover, for any fixed $d$, the orbits obtained from the action of $H_4$ on any point are congruent to each other.
\end{enumerate}
\end{proposition}
\begin{proof}
(1) At any point  $p\in M_{4,d}$, we consider the frame field $\{V_i\}^3_{i=1}$ as follows:
$$
V_1=p_{x_1}=e^{-\alpha x_2} E_1,\ \
V_2=p_{x_2}=\tfrac{d}{\sqrt{1-d^2}}E_2+E_4,\ \
V_3=p_{x_3}=e^{\alpha x_2}E_3.
$$
This shows that $N=-\sqrt{1-d^2} E_2 +d\, E_4$ is a unit normal vector field.

(2) Consider the orthonormal frame field  $\{W_i\}^3_{i=1}$ on $M_{4,d}$:
\begin{equation}\label{aaa:3.1}
W_1=\tfrac{V_1}{e^{-\alpha x_2}}=E_1,\ \ W_2=\sqrt{1-d^2}V_2=d E_2+\sqrt{1-d^2} E_4,\ \
W_3=\tfrac{V_3}{e^{\alpha x_2}}=E_3.
\end{equation}
Using \eqref{ttt:2.3}, \eqref{aaa:3.1}, $\beta=0$, $\gamma=-\alpha$, the Gauss and Weingarten formulae, we get
\begin{equation}\label{aaa:3.000}
\begin{aligned}
&\nabla_{W_1}W_1 =\sqrt{1-d^2} \alpha W_2,\ \ \ \  \nabla_{W_1}W_2 =-\sqrt{1-d^2} \alpha W_1,\\
&\nabla_{W_3}W_2=\sqrt{1-d^2} \alpha W_3,\ \ \ \
\nabla_{W_3}W_3 =-\sqrt{1-d^2} \alpha W_2,\\
&\nabla_{W_1}W_3=\nabla_{W_2}W_1=\nabla_{W_2}W_2=\nabla_{W_2}W_3=
\nabla_{W_3}W_1 =0,
\end{aligned}
\end{equation}
\begin{equation}\label{aaa:3.2}
AW_1=\alpha d\, W_1, \ AW_2=0,\
AW_3=-\alpha d\, W_3.
\end{equation}
It follows from \eqref{aaa:3.2} that the principal curvatures of $M_{4,d}$ are $\alpha d$, $0$ and $-\alpha d$.

By direct calculations, with the use of \eqref{aaa:3.000}, we obtain the sectional curvature and Ricci curvature of $M_{4,d}$ are
$$
K(W_1\wedge W_2)=K(W_2\wedge W_3)=\alpha^2(d^2-1),\ \ K(W_1\wedge W_3)=\alpha^2(1-d^2),
$$
$$
\mathrm{Ric}(W_1)=\mathrm{Ric}(W_3)=0,\ \
\mathrm{Ric}(W_2)=2\alpha^2(d^2-1) W_2.
$$

(3) Using \eqref{ttt:2.0}, we can directly verify that $H_4$ is a closed subgroup of $\mathrm{Sol}_{m,m}^4$. Moreover, it follows from \eqref{ttt:2.0} that the hypersurface $M_{4,d}$ is an orbit of the subgroup $H_4$ which passes through $(0,0,0,0)$,  and it has no focal manifold.

For any fixed $d$ and any point $(y_1,y_2,y_3,y_4)\in \mathrm{Sol}_{m,m}^4$,
put $H_4\cdot(0,0,0,0)=(x_1,\tfrac{d}{\sqrt{1-d^2}}x_2,x_3,x_2)=:M_1$ and
$H_4\cdot(y_1,y_2,y_3,y_4)=(x_1+e^{\alpha x_2}y_1,\tfrac{d}{\sqrt{1-d^2}}x_2+y_2,
x_3+e^{\gamma x_2}y_3,x_2+y_4)=:M_2$. After a left translation by $(0, -y_2, 0, -y_4)$
to the hypersurface $M_2$, we get $(0, -y_2, 0, -y_4)\cdot M_2=
(e^{-\alpha y_4}(x_1+e^{\alpha x_2}y_1),\tfrac{d}{\sqrt{1-d^2}}x_2,
e^{-\gamma y_4}(x_3+e^{\gamma x_2}y_3),x_2)$, which shows that $dt-\sqrt{1-d^2}y=0$ holds. Then from item (1),
we know that $M_1$ and $M_2$ are congruent.
\end{proof}

Finally, for any $m$ and $n$, we introduce the following three families of examples.

\begin{example}\label{exa:3.3}
For any given $r\geq0$, we define the hypersurface
\begin{equation*}
M_{5,r}:=\{(\tfrac{1}{\alpha}e^{\alpha x_3}\tanh (-\alpha r),x_1,x_2,x_3-\tfrac{1}{\alpha}\ln(\cosh (-\alpha r)))\in \mathrm{Sol}_{m,n}^4\mid x_{1},x_{2},x_{3}\in\mathbb{R}\}.
\end{equation*}
Put $H_5:=\{(0,x_1,x_2,x_3)\in \mathrm{Sol}_{m,n}^4\mid  x_{1},x_{2},x_{3}\in \mathbb R\}\subset\operatorname{Iso}_o(\mathrm{Sol}_{m,n}^4)=\mathrm{Sol}_{m,n}^4$.
\end{example}
\begin{proposition}\label{prop:3.3}
The hypersurface $M_{5,r}$ has the following properties:
\begin{enumerate}[{\rm(1)}]
\item The hypersurface $M_{5,r}$ can also be presented as $\{(x,y,z,t)\in \mathrm{Sol}_{m,n}^4\mid x e^{-\alpha t}=\tfrac{\sinh(-\alpha r)}{\alpha}\}$.
The unit normal vector field is given by $N=-\operatorname{sech}(-\alpha r)\, E_1 + \tanh(-\alpha r)\, E_4$, where $\{E_i\}^4_{i=1}$ is defined
by \eqref{ttt:2.1};
\item It is a minimal hypersurface and has constant principal curvatures $\alpha \tanh (-\alpha r)$, $\beta \tanh (-\alpha r)$ and $\gamma \tanh (-\alpha r)$. When $r=0$, the hypersurface $M_{5,0}$ is totally geodesic. The eigenvalues of the Ricci tensor of $M_{5,r}$ are $-(\beta^2+\gamma^2)\operatorname{sech}^2(-\alpha r)$,  $-\beta(\beta+\gamma)\operatorname{sech}^2(-\alpha r)$ and $-\gamma(\beta+\gamma)\operatorname{sech}^2(-\alpha r)$;

\item It is an orbit through $(\tfrac{1}{\alpha}\tanh(-\alpha r),0,0,-\tfrac{1}{\alpha}\ln(\cosh(-\alpha r)))$ of the subgroup $H_5$. Therefore, $M_{5,r}$ is a homogeneous hypersurface, and it has no focal manifold.
\end{enumerate}
\end{proposition}
\begin{proof}
(1) At any point  $p\in M_{5,r}$, we consider the frame field $\{V_i\}^3_{i=1}$ as follows:
$$
\begin{aligned}
V_1&=p_{x_1}=e^{-\beta (x_3-\tfrac{1}{\alpha}\ln(\cosh (-\alpha r)))}E_2, \ \
V_2=p_{x_2}=e^{-\gamma(x_3-\tfrac{1}{\alpha}\ln(\cosh (-\alpha r)))}E_3,\\
V_3&=p_{x_3}=\sinh(-\alpha r)\, E_1+E_4.
\end{aligned}
$$
Thus $N=-\operatorname{sech}(-\alpha r)\, E_1 + \tanh(-\alpha r)\, E_4$ is a unit normal vector field.

(2) Now, we choose the orthonormal frame field  $\{W_i\}^3_{i=1}$ on $M_{5,r}$:
\begin{equation}\label{bbb:3.1}
\begin{aligned}
W_1=\tfrac{V_3}{\cosh(-\alpha r)}=\tanh(-\alpha r)\, E_1+\operatorname{sech}(-\alpha r)\, E_4, \ \ W_2=E_2,\ \
W_3=E_3.
\end{aligned}
\end{equation}
Using \eqref{ttt:2.3}, \eqref{bbb:3.1}, the Gauss and  Weingarten formulae, we obtain
\begin{equation}\label{bbb:3.000}
\begin{aligned}
&\nabla_{W_1}W_1= \nabla_{W_1}W_2=\nabla_{W_1}W_3=\nabla_{W_2}W_3=\nabla_{W_3}W_2=0,\\
&\nabla_{W_2}W_1=-\operatorname{sech}(-\alpha r)\beta\, W_2,\ \
\nabla_{W_2}W_2 =\operatorname{sech}(-\alpha r) \beta\,  W_1,\\
&\nabla_{W_3}W_1=-\operatorname{sech}(-\alpha r) \gamma\, W_3,\ \
\nabla_{W_3}W_3=\operatorname{sech}(-\alpha r) \gamma W_1,
\end{aligned}
\end{equation}
\begin{equation}\label{bbb:3.2}
AW_1=\alpha \tanh(-\alpha r)\, W_1, \ \ AW_2=\beta\tanh(-\alpha r)W_2, \ \
AW_3=\gamma\tanh(-\alpha r)\, W_3.
\end{equation}
The equation \eqref{bbb:3.2} shows that the principal curvatures of $M_{5,r}$ are $\alpha \tanh (-\alpha r)$, $\beta \tanh (-\alpha r)$ and $\gamma \tanh (-\alpha r)$.

It follows from \eqref{bbb:3.000} that the sectional curvature and Ricci curvature of $M_{5,r}$ are
$$
\begin{aligned}
&K(W_1\wedge W_2)=-\beta^2\operatorname{sech}^2(-\alpha r),\ \ K(W_1\wedge W_3)=-\gamma^2\operatorname{sech}^2(-\alpha r),\\
&K(W_2\wedge W_3)=-\beta \gamma\operatorname{sech}^2(-\alpha r),\ \
\mathrm{Ric}(W_1)=-(\beta^2+\gamma^2)\operatorname{sech}^2(-\alpha r) W_1,\\
&\mathrm{Ric}(W_2)=-\beta(\beta+\gamma)\operatorname{sech}^2(-\alpha r) W_2,\ \
\mathrm{Ric}(W_3)=-\gamma(\beta+\gamma)\operatorname{sech}^2(-\alpha r) W_3.
\end{aligned}
$$

(3) According to \eqref{ttt:2.0}, we can directly verify that $H_5$ is a closed subgroup of $\mathrm{Sol}_{m, n}^4$. Moreover, it follows from \eqref{ttt:2.0} that the hypersurface $M_{5,r}$ is an orbit of the subgroup $H_5$ which passes through $(\tfrac{1}{\alpha}\tanh(-\alpha r),0,0,-\tfrac{1}{\alpha}\ln(\cosh(-\alpha r)))$,  and it has no focal manifold.
\end{proof}
\begin{example}\label{exa:3.5}
For any given $r\geq0$, we define the hypersurface
$$
M_{6,r}:=\{(x_1,x_2,\tfrac{1}{\gamma}e^{\gamma x_3}\tanh (\gamma r),x_3-\tfrac{1}{\gamma}\ln(\cosh (\gamma r)))\in \mathrm{Sol}_{m,n}^4\mid x_{1},x_{2},x_{3}\in\mathbb{R}\}.
$$
Put $H_6:=\{(x_1,x_2,0,x_3)\in \mathrm{Sol}_{m,n}^4\mid  x_1,x_2,x_3\in
\mathbb R\}\subset\operatorname{Iso}_o(\mathrm{Sol}_{m,n}^4)=\mathrm{Sol}_{m,n}^4$.
\end{example}
\begin{proposition}\label{prop:3.5}
The hypersurface $M_{6,r}$ has the following properties:
\begin{enumerate}[{\rm(1)}]
\item The hypersurface $M_{6,r}$ can also be presented as $\{(x,y,z,t)\in \mathrm{Sol}_{m, n}^4\mid z e^{-\gamma t}=\tfrac{\sinh(\gamma r)}{\gamma}\}$.
The unit normal vector field is given by $N=-\operatorname{sech}(\gamma r)\, E_3 + \tanh(\gamma r)\, E_4$, where $\{E_i\}^4_{i=1}$ is defined
by \eqref{ttt:2.1};
\item It is a minimal hypersurface and has constant principal curvatures $\alpha \tanh (\gamma r)$, $\beta \tanh (\gamma r)$ and $\gamma \tanh (\gamma r)$. When $r=0$, $M_{6,0}$ is a totally geodesic hypersurface. The eigenvalues of the Ricci tensor of $M_{6,r}$ are $-\alpha(\alpha+\beta)\operatorname{sech}^2(\gamma r)$,  $-\beta(\alpha+\beta)\operatorname{sech}^2(\gamma r)$ and $-(\alpha^2+\beta^2)\operatorname{sech}^2(\gamma r)$;

\item It is an orbit through $(0,0, \tfrac{1}{\gamma}\tanh(\gamma r),-\tfrac{1}{\gamma}\ln(\cosh(\gamma r)))$ of the subgroup $H_6$. Therefore, $M_{6,r}$ is a homogeneous hypersurface, and it has no focal manifold.
\end{enumerate}
\end{proposition}
\begin{proof}
(1) At any point  $p\in M_{6,r}$, we consider the frame field $\{V_i\}^3_{i=1}$ as follows:
$$
\begin{aligned}
V_1&=p_{x_1}=e^{-\alpha (x_3-\tfrac{1}{\gamma}\ln(\cosh (\gamma r)))}E_1,\ \
V_2=p_{x_2}=e^{-\beta(x_3-\tfrac{1}{\gamma}\ln(\cosh (\gamma r)))}E_2,\\
V_3&=p_{x_3}=\sinh(\gamma r)\, E_3+E_4.
\end{aligned}
$$
Thus, $N=-\operatorname{sech}(\gamma r)\, E_3 + \tanh(\gamma r)\, E_4$ is a unit normal vector field.

(2) Choose the orthonormal frame field  $\{W_i\}^3_{i=1}$ on $M_{6,r}$ as follows:
\begin{equation}\label{ccc:3.1}
W_1=E_1,\ \
W_2=E_2,\ \
W_3=\tfrac{V_3}{\cosh(\gamma r)}=\tanh(\gamma r)\, E_3+\operatorname{sech}(\gamma r)\, E_4.
\end{equation}
By using \eqref{ttt:2.3}, \eqref{ccc:3.1}, the Gauss and Weingarten formulae, we derive
\begin{equation}\label{ccc:3.000}
\begin{aligned}
&\nabla_{W_1}W_1=\operatorname{sech}(\gamma r)\alpha W_3,\ \ \ \
\nabla_{W_1}W_3=-\operatorname{sech}(\gamma r) \alpha W_1,\\
&\nabla_{W_2}W_2=\operatorname{sech}(\gamma r) \beta W_3,\ \ \ \
\nabla_{W_2}W_3=-\operatorname{sech}(\gamma r) \beta W_2, \\
&\nabla_{W_1}W_2=\nabla_{W_2}W_1=\nabla_{W_3}W_1=\nabla_{W_3}W_2=\nabla_{W_3}W_3=0,
\end{aligned}
\end{equation}
\begin{equation}\label{ccc:3.2}
AW_1=\alpha \tanh(\gamma r)\, W_1, \ AW_2=\beta\tanh(\gamma r)W_2,\
AW_3=\gamma\tanh(\gamma r)\, W_3.
\end{equation}
Then the principal curvatures of $M_{6,r}$ are $\alpha \tanh (\gamma r)$, $\beta \tanh (\gamma r)$ and $\gamma \tanh (\gamma r)$.

By \eqref{ccc:3.000}, we can obtain the sectional curvature and Ricci curvature
of $M_{6,r}$:
$$
\begin{aligned}
&K(W_1\wedge W_2)=-\alpha\beta\operatorname{sech}^2(\gamma r),\ \ K(W_1\wedge W_3)=-\alpha^2\operatorname{sech}^2(\gamma r),\\
&K(W_2\wedge W_3)=-\beta^2\operatorname{sech}^2(\gamma r),\ \
\mathrm{Ric}(W_1)=-\alpha(\alpha+\beta)\operatorname{sech}^2(\gamma r) W_1,\\
&\mathrm{Ric}(W_2)=-\beta(\alpha+\beta)\operatorname{sech}^2(\gamma r) W_2,\ \
\mathrm{Ric}(W_3)=-(\alpha^2+\beta^2)\operatorname{sech}^2(\gamma r) W_3.
\end{aligned}
$$

(3) From \eqref{ttt:2.0}, we know that $H_6$ is a closed subgroup of $\mathrm{Sol}_{m, n}^4$. Moreover, it follows from \eqref{ttt:2.0} that the hypersurface $M_{6,r}$ is an orbit of the subgroup $H_6$ which passes through the point $(0,0, \tfrac{1}{\gamma}\tanh(\gamma r),-\tfrac{1}{\gamma}\ln(\cosh(\gamma r)))$,  and it has no focal manifold.
\end{proof}

\begin{example}[cf. \cite{BM}]\label{exa:3.6}
For any given $r \in \mathbb{R}$, we consider the hypersurface $M_{7,r}$ defined by
$$
M_{7,r} := \{(x_1, x_2, x_3, r) \in \mathrm{Sol}_{m,n}^4 \mid x_1, x_2, x_3 \in \mathbb{R}\}.
$$
\end{example}
Let $H_7:= \{(x_1, x_2, x_3, 0) \in \mathrm{Sol}_{m,n}^4 \mid x_1, x_2, x_3 \in \mathbb{R}\}$. Then by \eqref{ttt:2.0}, we know that $H_7$ is a closed subgroup of $\mathrm{Sol}_{m, n}^4$. Thus, by using \eqref{ttt:2.0}, \eqref{ttt:2.1}, \eqref{ttt:2.3}, the Gauss and Weingarten formulae,
and the definition of Riemannian curvature tensor, we have the following proposition without proof.
\begin{proposition}\label{prop:3.6}
The hypersurface $M_{7,r}$ has the following properties:
\begin{enumerate}[\rm(1)]
    \item The hypersurface $M_{7,r}$ can also be presented as $\{(x,y,z,t)\in \mathrm{Sol}_{m,n}^4\mid t=r\}$. The unit normal vector field is given by $N = E_4$, where $E_4$ is defined by \eqref{ttt:2.1};
    \item It is a minimal hypersurface and has three distinct constant principal curvatures $\alpha$, $\beta$ and $\gamma$. It is also a flat hypersurface;
    \item It is an orbit of subgroup $H_7$ which passes through $(0, 0, 0, r)$. Therefore, the hypersurface $M_{7,r}$ is homogeneous, and it has no focal manifold.
        Moreover, for any given $r$, the hypersurface $M_{7,r}$ is congruent to $M_{7,0}$.
\end{enumerate}
\end{proposition}

\subsection{Examples of homogeneous hypersurfaces of $\mathrm{Nil}^4$}\label{sect:3.3}~
\begin{example}\label{exa:3.11}
For any fixed $0\leq d<1$, we define the hypersurface
$$
M_{8,d}:=\big\{(x_1,x_2,\tfrac{d}{\sqrt{1-d^2}}x_3,x_3)\in \mathrm{Nil}^4\mid x_{1},x_{2},x_{3}\in\mathbb{R}\big\}.
$$
Put $H_8:=\{(x_1,x_2,\tfrac{d}{\sqrt{1-d^2}}x_3,x_3)\in \mathrm{Nil}^4\mid  x_{1},x_{2},x_{3}\in \mathbb R\}\subset\operatorname{Iso}_o(\mathrm{Nil}^4)=\mathrm{Nil}^4$.
\end{example}
\begin{proposition}\label{prop:3.11}
The hypersurface $M_{8,d}$ has the following properties:
\begin{enumerate}[{\rm(1)}]
\item The hypersurface $M_{8,d}$ can also be presented as $\{(x,y,z,t)\in \mathrm{Nil}^4\mid \sqrt{1-d^2}z-dt=0\}$.
The unit normal vector field is $N=-\sqrt{1-d^2}\, E_3+ d\, E_4$, where $\{E_i\}^4_{i=1}$ is defined
by \eqref{hhh:2.1};
\item It is a minimal hypersurface and has three distinct constant principal curvatures $-\frac{\sqrt{1+d^2}}{2}$, $0$ and  $\frac{\sqrt{1+d^2}}{2}$. The eigenvalues of the Ricci tensor of $M_{8,d}$ are $\frac{1-d^2}{2}$ and $\frac{d^2-1}{2}$ (multiplicity 2);
\item It is an orbit of the subgroup $H_8$ which passes through $(0,0,0,0)$. Therefore, the hypersurface $M_{8,d}$ is homogeneous, and it has no focal manifold.
    Moreover, for any fixed $d$, the orbits obtained from the action of $H_8$ on any point are congruent to each other.
\end{enumerate}
\end{proposition}
\begin{proof}
(1) At any point  $p\in M_{8,d}$, we consider the frame field $\{V_i\}^3_{i=1}$ as follows:
$$
\begin{aligned}
V_1&=p_{x_1}=E_1, \ \ V_2=p_{x_2}=E_2-x_3\, E_1, \\
V_3&=p_{x_3}=\tfrac{d}{2\sqrt{1-d^2}}x^2_3\, E_1-\tfrac{d}{\sqrt{1-d^2}}x_3\, E_2+\tfrac{d}{\sqrt{1-d^2}}\, E_3+E_4.
\end{aligned}
$$
Hence $N=-\sqrt{1-d^2}E_3 + d E_4$ is a unit normal vector field on $M$.

(2) Now, we choose the orthonormal frame field  $\{W_i\}^3_{i=1}$ on $M_{8,d}$ as follows:
\begin{equation}\label{sss:3.1}
W_1=E_1, \ \ W_2=E_2,\ \
W_3=d\, E_3+\sqrt{1-d^2}E_4.
\end{equation}
It follows from \eqref{hhh:2.3}, \eqref{sss:3.1}, the Gauss and Weingarten formulae that
\begin{equation}\label{sss:3.000}
\begin{aligned}
\nabla_{W_1}W_1 &=0,      & \nabla_{W_1}W_2 &= \tfrac{\sqrt{1-d^2}}{2}W_3,
& \nabla_{W_1}W_3 &=-\tfrac{\sqrt{1-d^2}}{2}W_2,\\
\nabla_{W_2}W_1 &=\tfrac{\sqrt{1-d^2}}{2}W_3        ,       & \nabla_{W_2}W_2 &=0,
&\nabla_{W_2}W_3 &=-\tfrac{\sqrt{1-d^2}}{2}W_1, \\
\nabla_{W_3}W_1 &=-\tfrac{\sqrt{1-d^2}}{2}W_2,       & \nabla_{W_3}W_2&=\tfrac{\sqrt{1-d^2}}{2}W_1,
& \nabla_{W_3}W_3 &=0,
\end{aligned}
\end{equation}
\begin{equation}\label{sss:3.2}
AW_1=\tfrac{d}{2}W_2, \ \ AW_2=\tfrac{d}{2}W_1+\tfrac{1}{2}W_3,\ \
AW_3=\tfrac{1}{2}W_2.
\end{equation}
Thus, the principal curvatures of $M_{8,d}$ are $-\frac{\sqrt{1+d^2}}{2}$, $0$ and  $\frac{\sqrt{1+d^2}}{2}$.

By \eqref{sss:3.000}, we can calculate the sectional curvature and Ricci curvature of $M_{8,d}$ to obtain
$$
\begin{aligned}
&K(W_1\wedge W_2)=K(W_1\wedge W_3)=\tfrac{1-d^2}{4},\ \ K(W_2\wedge W_3)=\tfrac{3(d^2-1)}{4},\\
&\mathrm{Ric}(W_1)=\tfrac{1-d^2}{2} W_1,\ \ \mathrm{Ric}(W_2)=\tfrac{d^2-1}{2} W_2,\ \ \mathrm{Ric}(W_3)=\tfrac{d^2-1}{2} W_3.
\end{aligned}
$$

(3) From \eqref{hhh:2.0}, we can directly verify that $H_8$ is a closed subgroup of $\mathrm{Nil}^4$. Moreover, according to \eqref{hhh:2.0}, we further derive that the hypersurface $M_{8,d}$ is an orbit of the subgroup $H_8$ which passes through $(0,0,0,0)$,  and it has no focal manifold.

For any fixed $d$ and any point $(y_1,y_2,y_3,y_4)\in \mathrm{Nil}^4$,
put $H_8\cdot(0,0,0,0)=(x_1,x_2,\tfrac{d}{\sqrt{1-d^2}}x_3,x_3)=:M_3$ and
$H_8\cdot(y_1,y_2,y_3,y_4)=(x_1+y_1+x_3 y_2+\tfrac{x_3^2 y_3}{2},x_2+y_2+x_3 y_3,\tfrac{d}{\sqrt{1-d^2}}x_3+y_3,x_3+y_4)=:M_4$.
Similar to the proof of item (3) of Proposition \ref{prop:3.8}, one can verify that
up to a left translation by $(0,0,-y_3,-y_4)$, the hypersurfaces $M_3$ and $M_4$ are congruent.
\end{proof}

\begin{example}[cf. Remark 3.1 of \cite{DHCB}]\label{exa:3.12}
For any given $r \in \mathbb{R}$, we consider the hypersurface defined by
$$
M_{9,r} := \{(x_1, x_2, x_3, r) \in \mathrm{Nil}^4 \mid x_1, x_2, x_3 \in \mathbb{R}\}.
$$
\end{example}
Let $H_9:= \{(x_1, x_2, x_3, 0) \in \mathrm{Nil}^4 \mid x_1, x_2, x_3 \in \mathbb{R}\}$. Then it follows from \eqref{hhh:2.0} that $H_9$ is a closed subgroup of $\mathrm{Nil}^4$. Using \eqref{hhh:2.0}, \eqref{hhh:2.1}, \eqref{hhh:2.3}, the Gauss and Weingarten formulae, and the definition of Riemannian curvature tensor, we have the following proposition without proof.
\begin{proposition}\label{prop:3.12}
The hypersurface $M_{9,r}$ has the following properties:
\begin{enumerate}[\rm(1)]
    \item The hypersurface $M_{9,r}$ can also be presented as $\{(x,y,z,t)\in \mathrm{Nil}^4\mid t=r\}$. The unit normal vector field is given by $N = E_4$, where $E_4$ is defined by \eqref{hhh:2.1};
    \item It is a minimal hypersurface and has three distinct constant principal curvatures $-\tfrac{1}{\sqrt{2}}$, $0$ and $\tfrac{1}{\sqrt{2}}$. It is also a flat hypersurface;
    \item It is an orbit of subgroup $H_9$ which passes through $(0, 0, 0, r)$. Hence, $M_{9,r}$ is a homogeneous hypersurface, and it has no focal manifold.
        Moreover, for any given $r$, the hypersurface $M_{9,r}$ is congruent to $M_{9,0}$.
\end{enumerate}
\end{proposition}
%

\section{ Proof of Theorem \ref{thm:1.2}}\label{sect:4}
In order to prove Theorem \ref{thm:1.2}, we first give the following theorem.

\begin{theorem}\label{thm:1.1}
Let $M$ be a hypersurface of $\mathrm{Sol}_1^4$ with constant angle functions $a$, $b$, $c$ and $d$. Then, up to isometries of $\mathrm{Sol}_1^4$, one of the following two cases occurs:
\begin{enumerate}[\rm(1)]
  \item $M$ is an open part of $M_{1,r}$ for some $r\geq0$;
  \item $M$ is an open part of $M_{2,0}$.
\end{enumerate}
\end{theorem}

\begin{proof}
Let $M$ be a hypersurface of $\mathrm{Sol}_1^4$ with the unit normal vector field $N = aE_1 + bE_2 + cE_3 + dE_4$, where $\{E_i\}_{i=1}^4$ is defined by \eqref{kkk:2.1} and $a$, $b$, $c$, $d$ are constants. It follows from $a^2+b^2+c^2+d^2=1$ that $|d|\leq1$. Moreover, up to changing the sign of the unit normal vector field, we can assume that $d\geq0$.

Since $0\leq d\leq1$ is a constant, we can consider the following two cases:

\vskip2mm
{\bf Case I.} $0\leq d<1$ on $M$.

In this case, we choose the orthonormal frame field $\{T_i\}_{i=1}^3$ as defined by \eqref{kkk:2.5}.
According to \eqref{kkk:2.3}, the Weingarten formula and $a$, $b$, $c$, $d$ are constants, we obtain
\begin{equation*}
\begin{aligned}
AT_1&=-\tilde{\nabla}_{T_1}N=-\tilde{\nabla}_{bE_1-aE_2+dE_3-cE_4}(aE_1+bE_2+cE_3+dE_4)\\
&=-ab\tilde{\nabla}_{E_1}E_1-b^2\tilde{\nabla}_{E_1}E_2
-bc\tilde{\nabla}_{E_1}E_3-bd\tilde{\nabla}_{E_1}E_4
+a^2\tilde{\nabla}_{E_2}E_1+ab\tilde{\nabla}_{E_2}E_2\\
& \ \ \
+ac\tilde{\nabla}_{E_2}E_3+ad\tilde{\nabla}_{E_2}E_4
-ad\tilde{\nabla}_{E_3}E_1-bd\tilde{\nabla}_{E_3}E_2
-cd\tilde{\nabla}_{E_3}E_3-d^2\tilde{\nabla}_{E_3}E_4\\
& \ \ \
+ac\tilde{\nabla}_{E_4}E_1+bc\tilde{\nabla}_{E_4}E_2
+c^2\tilde{\nabla}_{E_4}E_3+cd\tilde{\nabla}_{E_4}E_4\\
&=\tfrac{1}{2}(ac+bd)E_1+\tfrac{1}{2}(bc+3ad)E_2-\tfrac{1}{2}(a^2+b^2)E_3-2abE_4\\
&=2a(bc-ad)T_1+\tfrac{1}{2}a(a^2-3 b^2+c^2-3d^2)T_2+\tfrac{1}{2}(5a^2b+4acd +b(b^2+c^2+d^2))T_3.
\end{aligned}
\end{equation*}
Similar calculations give that
\begin{equation*}
\begin{aligned}
AT_2& = \tfrac{1}{2}(a^3+4bcd +a(b^2+c^2-3d^2))T_1+d(a^2-b^2+c^2-d^2)T_2 \\
&\ \ \ \  +\tfrac{1}{2}(c(a^2+b^2+c^2)+4abd +5cd^2)T_3,\\
AT_3 &= \tfrac{1}{2}(a^2b+4acd+b(b^2-3c^2+d^2))T_1-\tfrac{1}{2}c(a^2-3b^2+c^2-3d^2)T_2\\
&\ \ \ \  +(-2 abc+d^3+d(a^2+b^2-c^2))T_3.
\end{aligned}
\end{equation*}
Then, it follows from the symmetry of the shape operator $A$ that
\begin{align}
-2b(cd+ab)&=0,\label{kkk:4.2}\\
2b(a^2+c^2)&=0,\label{kkk:4.3}\\
c(a^2+c^2+d^2-b^2)+2abd&=0.\label{kkk:4.4}
\end{align}

If $b\neq0$ on $M$, from \eqref{kkk:4.3} we derive that $a=c=0$.
If $b=0$ on $M$, from \eqref{kkk:4.4} and $a^2+c^2+d^2=1$, we get $c=0$.
Therefore, we know that the solutions to \eqref{kkk:4.2}--\eqref{kkk:4.4} are $a=c=0$ or $b=c=0$.

If $a=c=0$ holds, then we have $-N=-b E_2-d E_4$.
By using \eqref{fff:2.4}, we know that the unit normal vector field of the hypersurface $\phi_4(M)$ is $-b E_1+d E_4$. Therefore, up to isometries of $\mathrm{Sol}_1^4$, we only need to consider the case $b=c=0$.

Since $b=c=0$, we have $N=aE_1+dE_4$, where $a^2+d^2=1$.
Up to the action of $\phi_1\in D_4$ defined by \eqref{fff:2.4},
we can always assume that $a=-\sqrt{1-d^2}$.

We consider the orthonormal frame field $\{W_i\}_{i=1}^3$ on $M$ defined by
\begin{equation}\label{kkk:4.5}
W_1=d E_1+\sqrt{1-d^2} E_4,\ \ W_2=E_2,\ \ W_3=E_3.
\end{equation}
By applying \eqref{kkk:2.3}, \eqref{kkk:2.6}, \eqref{kkk:4.5}, and $d$ is a constant, we obtain
\begin{equation}\label{kkk:4.7}
\begin{aligned}
\nabla_{W_1}W_1 &=0, & \nabla_{W_1}W_2 &= \tfrac{d}{2}W_3, &\nabla_{W_1}W_3 &=-\tfrac{d}{2}W_2,\\
\nabla_{W_2}W_1 &=\sqrt{1-d^2} W_2-\tfrac{d}{2}W_3, & \nabla_{W_2}W_2 &=-\sqrt{1-d^2} W_1,
&\nabla_{W_2}W_3 &=\tfrac{d}{2}W_1,\\
\nabla_{W_3}W_1&=-\tfrac{d}{2}W_2, & \nabla_{W_3}W_2&=\tfrac{d}{2}W_1,
&\nabla_{W_3}W_3&=0,
\end{aligned}
\end{equation}
\begin{equation}\label{zx:1.8}
AW_1 =dW_1,       \ \ AW_2= -dW_2-\tfrac{\sqrt{1-d^2}}{2}W_3,       \ \ AW_3 = -\tfrac{\sqrt{1-d^2}}{2}W_2.
\end{equation}

It follows from \eqref{kkk:4.7} that
\begin{equation}\label{kkk:4.8}
[W_1,W_2]=-\sqrt{1-d^2} W_2 +dW_3,\ \  [W_1,W_3]=0,\ \ [W_2,W_3]=0.
\end{equation}
Then we can check that all the Gauss and Codazzi equations are satisfied.

From \eqref{kkk:2.1} and \eqref{kkk:4.5}, we deduce
\begin{equation}\label{kkk:4.6}
W_1(x)=e^t d, \ \ W_2(x)=W_3(x)=0, \ \ W_1(t)=\sqrt{1-d^2},\ \ W_2(t)=W_3(t)=0.
\end{equation}

Now, we define a new frame field $\{X_i\}_{i=1}^3$ on $M$ as follows:
\begin{equation}\label{kkk:4.9}
X_1=W_1,\ \ X_2=e^t\, W_2-x\, W_3,\ \ X_3=W_3.
\end{equation}
By using \eqref{kkk:4.8}, \eqref{kkk:4.6} and \eqref{kkk:4.9}, we derive that $[X_i,X_j]=0$, $1\leq i<j\leq3$.

Then, we can locally identify $M$ with an open subset $\Omega$ of $\mathbb R^3$ and express the
hypersurface $M$ by an immersion
$$
\begin{aligned}
\Phi:\mathbb R^3\supset\Omega& \longrightarrow \mathrm{Sol}_1^4 ,\\
(u,v,w)&\longmapsto(x(u,v,w),y(u,v,w),z(u,v,w),t(u,v,w)),
\end{aligned}
$$
such that
\begin{equation}\label{kkk:4.10}
\begin{aligned}
&\mathrm{d} \Phi(\partial_u)
=\left(x_u,y_u,z_u,t_u\right)=X_1,\\
&\mathrm{d} \Phi(\partial_v)
=\left(x_v,y_v,z_v,t_v\right)=X_2,\\
&\mathrm{d} \Phi(\partial_w)
=\left(x_w,y_w,z_w,t_w\right)=X_3.
\end{aligned}
\end{equation}
According to \eqref{kkk:2.1}, \eqref{kkk:4.5} and \eqref{kkk:4.9}, we also have
\begin{equation}\label{kkk:4.11}
X_1=(d e^t,0,0,\sqrt{1-d^2}),\ \ X_2=(0,1,0,0),\ \  X_3=(0,0,1,0).
\end{equation}
The equations \eqref{kkk:4.10} and \eqref{kkk:4.11} show that $x$ and $t$ depend only on $u$, $y$ depends only on $v$, and $z$ depends only on $w$.

It follows from \eqref{kkk:4.10} and \eqref{kkk:4.11} that $y=v+c_1$, $z=w+c_2$, $t=\sqrt{1-d^2}u+c_3$, where $c_1$, $c_2$ and $c_3$ are constants. Thus, we have
$u=\tfrac{t-c_3}{\sqrt{1-d^2}}$ and
\begin{equation}\label{kkk:4.12}
u_t=\tfrac{1}{\sqrt{1-d^2}}.
\end{equation}
In the following, we use $t$ instead of $u$ as a local coordinate.

Then, from  \eqref{kkk:4.10}, \eqref{kkk:4.11} and \eqref{kkk:4.12}, we derive that $x_t=x_uu_t=\tfrac{d}{\sqrt{1-d^2}}e^t$, which implies that $x=\tfrac{d}{\sqrt{1-d^2}}e^t+c_4$, where $c_4$ is a constant.

Using $y$ and $z$ as the new local coordinates, we obtain $\Phi(t,y,z)=(\tfrac{d}{\sqrt{1-d^2}}e^t+c_4,y,z,t)$.
After a left translation by $(-c_4, 0, 0, 0)$, we get
$\Phi(t,y,z)=(\tfrac{d}{\sqrt{1-d^2}}e^t,y,z-c_4y,t)$.
Denoting $z-c_4y$ by $\tilde{z}$, we have
\begin{equation*}
\Phi(t,y,\tilde{z})=(\tfrac{d}{\sqrt{1-d^2}}e^t,y,\tilde{z},t).
\end{equation*}
Since $0\leq d< 1$, we may assume that $d=\tanh r$ for some constant $r\geq 0$. By taking the
reparametrization $x_1=y$, $x_2=\tilde{z}$ and $x_3=t+\ln(\cosh r)$, we get
$$
\Phi(x_1,x_2,x_3)=(e^{x_3}\tanh r,x_1,x_2,x_3-\ln(\cosh r)).
$$
This indicates that, up to isometries of $\mathrm{Sol}_1^4$, $M$ is an open part of $M_{1,r}$ for some $r\geq0$.

\vskip2mm
{\bf Case II.} $d=1$ on $M$.

In this case, the unit normal vector field is $N=E_4$, and $TM=\text{span}\{E_1, E_2, E_3\}$, where $\{E_i\}_{i=1}^4$ is defined by \eqref{kkk:2.1}. It follows from \eqref{kkk:2.0}, \eqref{kkk:2.1} and
\eqref{kkk:2.2} that $M$ is an open part of $\{(x_1,x_2,x_3,r) \in \mathrm{Sol}_1^4 \mid x_1,x_2,x_3 \in \mathbb{R}\}$ for some $r \in \mathbb{R}$, which is congruent to $M_{2,0}$.

\vskip1mm
Conversely, from Propositions \ref{prop:3.1} and \ref{prop:3.2}, we know that the hypersurfaces $M_{1,r}$ and $M_{2,0}$ have constant angle functions.

In conclusion, we have completed the proof of Theorem \ref{thm:1.1}.
\end{proof}

Now, we give the proof of Theorem \ref{thm:1.2}.

\noindent
{\bf Proof of Theorem \ref{thm:1.2}.}

Let $M$ be a homogeneous hypersurface of $\mathrm{Sol}_1^4$, then $M$ is an orbit of some closed subgroup $ G \subset \operatorname{Iso}_o(\mathrm{Sol}_1^4)=\mathrm{Sol}_1^4$. Let $N$ be a unit normal vector field of $M$ and fix a point $p_0\in M$. Since $M$ is homogeneous, for any $p\in M$, there exists an isometry $\phi \in \mathrm{Sol}_1^4$ such that $\phi(M)=M$ and $\phi(p_0)=p$. Then we know that $N_p = \pm \mathrm{d}\phi_{p_0}(N_{p_0})$. Since $E_i$ are left invariant vector fields on $\mathrm{Sol}_1^4$ for $i=1,2,3,4$, we have $\mathrm{d}\phi_{p_0}(E_i|_{p_0})=E_i|_p, i=1,2,3,4$.

According to the definition of angle functions, we get
$$
\begin{aligned}
a(p) &= g(N_p, E_1|_{p}) = g(\pm \mathrm{d}\phi_{p_0}(N_{p_0}), \mathrm{d}\phi_{p_0}(E_1|_{p_0})) = \pm a(p_0), \\
b(p) &= g(N_p, E_2|_{p}) = g(\pm \mathrm{d}\phi_{p_0}(N_{p_0}), \mathrm{d}\phi_{p_0}(E_2|_{p_0})) = \pm b(p_0), \\
c(p) &= g(N_p, E_3|_{p}) = g(\pm \mathrm{d}\phi_{p_0}(N_{p_0}), \mathrm{d}\phi_{p_0}(E_3|_{p_0})) = \pm c(p_0), \\
d(p) &= g(N_p, E_4|_{p}) = g(\pm \mathrm{d}\phi_{p_0}(N_{p_0}), \mathrm{d}\phi_{p_0}(E_4|_{p_0})) = \pm d(p_0).
\end{aligned}
$$
These equations together with the connectedness of $M$
show that the angle functions $a$, $b$, $c$ and $d$ are constants on $M$.
Then by Theorem \ref{thm:1.1}, up to isometries of $\mathrm{Sol}_1^4$, we know that $M$ is either $M_{1,r}$ for some $ r\geq0$ or $M_{2,0}$.

Conversely, from Propositions \ref{prop:3.1} and \ref{prop:3.2}, we know that both the hypersurfaces $M_{1,r}$ for some $ r\geq0$ and $M_{2,0}$ are homogeneous.

In conclusion, we have completed the  proof of Theorem \ref{thm:1.2}. \qed

\section{ Proof of Theorem \ref{thm:1.4}}\label{sect:5}

In order to prove Theorem \ref{thm:1.4}, we give the following theorem.

\begin{theorem}\label{thm:1.3}
Let $M$ be a hypersurface of $\mathrm{Sol}_{m,n}^4$ with constant angle functions $a$, $b$, $c$ and $d$. Then up to isometries of $\mathrm{Sol}_{m,n}^4$, one of the following five cases occurs:
\begin{enumerate}[\rm(1)]
\item $M$ is an open part of $M_{3,r}$ for some $r\geq0$;
\item $M$ is an open part of $M_{4,d}$ for some $0\leq d<1$;
\item $M$ is an open part of $M_{5,r}$ for some $r\geq0$;
\item $M$ is an open part of $M_{6,r}$ for some $r\geq0$;
\item $M$ is an open part of $M_{7,0}$.
\end{enumerate}
\end{theorem}

\begin{proof}
Suppose that $M$ is a hypersurface of $\mathrm{Sol}_{m,n}^4$ with
the unit normal vector field $N = aE_1 + bE_2 + cE_3 + dE_4$,
where $\{E_i\}_{i=1}^4$ is defined by \eqref{ttt:2.1} and
$a$, $b$, $c$, $d$ are constants.
It follows from $a^2+b^2+c^2+d^2=1$ that $|d|\leq1$.
Moreover, up to changing the sign of the unit normal
vector field, we may always set that $d\geq0$.
Then we consider the following two cases:

\vskip2mm
{\bf Case I.} $0\leq d<1$ on $M$.

In this case, we choose the orthonormal frame field $\{T_i\}_{i=1}^3$ as defined by \eqref{kkk:2.5}.
From \eqref{ttt:2.3}, the Weingarten formula, $a$, $b$, $c$ and $d$ are constants, we derive that
\begin{equation*}
\begin{aligned}
AT_1&=-\tilde{\nabla}_{T_1}N=-\tilde{\nabla}_{bE_1-aE_2+dE_3-cE_4}(aE_1+bE_2+cE_3+dE_4)\\
&=-ab\tilde{\nabla}_{E_1}E_1-b^2\tilde{\nabla}_{E_1}E_2
-bc\tilde{\nabla}_{E_1}E_3-bd\tilde{\nabla}_{E_1}E_4
+a^2\tilde{\nabla}_{E_2}E_1+ab\tilde{\nabla}_{E_2}E_2\\
& \ \ \
+ac\tilde{\nabla}_{E_2}E_3+ad\tilde{\nabla}_{E_2}E_4
-ad\tilde{\nabla}_{E_3}E_1-bd\tilde{\nabla}_{E_3}E_2
-cd\tilde{\nabla}_{E_3}E_3-d^2\tilde{\nabla}_{E_3}E_4\\
& \ \ \
+ac\tilde{\nabla}_{E_4}E_1+bc\tilde{\nabla}_{E_4}E_2
+c^2\tilde{\nabla}_{E_4}E_3+cd\tilde{\nabla}_{E_4}E_4\\
&=bd\alpha E_1-ad\beta E_2+d^2\gamma E_3+(ab(\beta-\alpha)-cd\gamma)E_4\\
&=(b^2 d \alpha + a b c (\alpha - \beta) + a^2 d \beta + d (c^2 + d^2) \gamma)T_1+(a b^2 (\beta-\alpha) + b c d (\alpha - \gamma) + a d^2 (\beta - \gamma))T_2\\
& \ \ \
+(a^2 b (\alpha - \beta) + b d^2 (\alpha - \gamma) + a c d (\gamma-\beta))T_3.
\end{aligned}
\end{equation*}
Similar calculations show that
\begin{equation*}
\begin{aligned}
AT_2& = (b c d (\alpha - \beta) + a c^2 (\alpha - \gamma) + a d^2 (\beta - \gamma))T_1\\
 &\ \ \ \  +(c^2 d \alpha + d (b^2 + d^2) \beta + a^2 d \gamma + a b c (\gamma-\alpha))T_2 \\
&\ \ \ \  +(c d^2 (\alpha - \beta) + a^2 c (\alpha - \gamma) + a b d (\gamma-\beta))T_3,\\
AT_3 &= (a c d (\alpha - \beta) + b d^2 (\alpha - \gamma) + b c^2 (\beta - \gamma))T_1\\
 &\ \ \ \  +(c d^2 (\alpha - \beta) + a b d (\gamma-\alpha) + b^2 c (\gamma-\beta))T_2 \\
&\ \ \ \  +(a^2 d \alpha + a b c (\beta - \gamma) + d (d^2 \alpha + c^2 \beta + b^2 \gamma))T_3.
\end{aligned}
\end{equation*}
Then, using the symmetry of the shape operator $A$ and $\beta=-\alpha-\gamma$, we obtain
\begin{align}
(-b c d - a c^2 - 2 a b^2) \alpha + (-2 b c d + a c^2 - a b^2) \gamma&=0,\label{ttt:4.2}\\
\left( -a c d + 2 a^2 b + b c^2 \right) \alpha + \left( a c d + a^2 b + 2 b c^2 \right) \gamma&=0,\label{ttt:4.3}\\
\left( a^2 c + 2 a b d - b^2 c \right) \alpha + \left( -a^2 c + a b d - 2 b^2 c \right) \gamma&=0.\label{ttt:4.4}
\end{align}
Put
\begin{equation}\label{uuuu:4.1}
\begin{aligned}
f_1:&=-b c d - a c^2 - 2 a b^2, & f_2&:= -2 b c d + a c^2 - a b^2,\\
f_3:&= -a c d + 2 a^2 b + b c^2, & f_4&:=a c d + a^2 b + 2 b c^2,\\
f_5:&=a^2 c + 2 a b d - b^2 c, &  f_6&:= -a^2 c + a b d - 2 b^2 c.
\end{aligned}
\end{equation}

We claim that $f_i=0$ if and only if $f_{i+1}=0$ for $i=1,3,5$. Indeed,
it follows from $\alpha+\beta+\gamma=0$ and $\alpha<\beta<\gamma$ that $\alpha<0$ and $\gamma>0$. These combining with the equations \eqref{ttt:4.2}--\eqref{uuuu:4.1} can derive this assertion.

We then claim that $f_3=0$. Indeed, if $f_3\neq0$ holds, then from \eqref{ttt:4.3}, we have $\alpha=-\frac{f_4}{f_3}\gamma$. If we assume that $f_1\neq0$ holds, then from \eqref{ttt:4.2}, we also have $\alpha=-\frac{f_2}{f_1}\gamma$. Since $\gamma>0$, we get
$$
0=f_1f_4-f_2f_3=-3 a b c^2 (a^2 + b^2 + c^2 + d^2)=-3 a b c^2,
$$
which shows that $abc=0$. But this cannot occur. In fact, if $ab\neq0$, then $c=0$.
These together with \eqref{ttt:4.2} give that $2\alpha+\gamma=0$, which contradicts $\alpha+\beta+\gamma=0$ and $\alpha<\beta<\gamma$.
Then we have $ab=0$.
If $a\neq0$, then $b=0$. These combining with \eqref{ttt:4.2} and $\alpha\neq\gamma$ give that $c=0$, which implies that $f_1=0$, a contradiction. Thus we have $a=0$.
The equation \eqref{ttt:4.2} becomes $(\alpha+2\gamma)bcd=0$.
Since $\alpha+2\gamma\neq0$, we get $bcd=0$. It follows from $a=0$ and $bcd=0$ that $f_1=0$, a contradiction. Thus we know that $f_1=0$ holds.
Then we have $f_2=0$. It follows from $4f_1-2f_2=0$
that $a(b^2+c^2)=0$. But this cannot also occur. If $a\neq0$, then we have $b=c=0$. Substituting this into \eqref{uuuu:4.1}, we get $f_3=0$, a contradiction. Thus we know that $a=0$. From $a=0$, $\alpha+2\gamma\neq0$ and \eqref{ttt:4.3}, we get
$bc=0$. It follows from $a=0$ and $bc=0$ that $f_3=0$, a contradiction.
Given the above, we get $f_3=0$.

It follows from $f_3=0$ that $f_4=0$.
Calculating $f_3+f_4=0$ and using \eqref{uuuu:4.1},
we get $b(a^2+c^2)=0$, which implies that $b=0$ or $a=c=0$.  When $b=0$, from $\alpha\neq\gamma$ and \eqref{ttt:4.2}, we get $ac=0$, which yields that $a=0$ or $c=0$.

In summary, we know that the solutions of \eqref{ttt:4.2}--\eqref{ttt:4.4} are $a=c=0$, $b=c=0$ or $a=b=0$, then we further consider the following three subcases:

\vskip2mm
{\bf Case I-(1).} $a=c=0$ on $M$.

In this subcase, we have $N=bE_2+dE_4$, where $b^2+d^2=1$.
Up to the action of reflection in the $y$-coordinate,
we assume that $b=-\sqrt{1-d^2}$.

On $M$, we consider the orthonormal frame field $\{W_i\}_{i=1}^3$:
\begin{equation}\label{ppp:4.5}
W_1=E_1,\ \ W_2=d E_2+\sqrt{1-d^2}E_4,\ \ W_3=E_3.
\end{equation}
By using \eqref{ttt:2.3}, \eqref{kkk:2.6}, \eqref{ppp:4.5} and $d$ is a constant,
we obtain
\begin{equation}\label{ppp:4.7}
\begin{aligned}
\nabla_{W_1}W_1 &=\sqrt{1-d^2} \alpha W_2,  \ \ \nabla_{W_1}W_2 =-\sqrt{1-d^2} \alpha W_1,\\
\nabla_{W_3}W_2&=-\sqrt{1-d^2} \gamma W_3,
\ \ \nabla_{W_3}W_3 =\sqrt{1-d^2} \gamma W_2,\\
\nabla_{W_1}W_3 &=\nabla_{W_2}W_1=\nabla_{W_2}W_2=\nabla_{W_2}W_3=\nabla_{W_3}W_1=0,
\end{aligned}
\end{equation}
\begin{equation}\label{zx1.1}
AW_1=d \alpha W_1,       \ \ AW_2= d \beta W_2,      \ \ AW_3 =d \gamma W_3.
\end{equation}
It follows from \eqref{ppp:4.7} that
\begin{equation}\label{ppp:4.8}
[W_1,W_2]=-\sqrt{1-d^2} \alpha W_1,\ \ [W_1,W_3]=0,\ \ [W_2,W_3]=\sqrt{1-d^2} \gamma W_3.
\end{equation}
Then it can be checked that all the Gauss and Codazzi equations are satisfied.

From \eqref{ttt:2.1} and \eqref{ppp:4.5}, we deduce
\begin{equation}\label{ppp:4.6}
W_1(t)=W_3(t)=0,\ \ W_2(t)=\sqrt{1-d^2}.
\end{equation}

Now, we define a new frame field $\{X_i\}_{i=1}^3$ on $M$ as follows:
\begin{equation}\label{ppp:4.9}
X_1=e^{-\alpha t}W_1,\ \ X_2=W_2,\ \   X_3=e^{-\gamma t}\,W_3.
\end{equation}
By using \eqref{ppp:4.8}, \eqref{ppp:4.6} and \eqref{ppp:4.9}, we derive that $[X_i,X_j]=0$, $1\leq i<j\leq3$.

Then, we can locally identify $M$ with an open subset $\Omega$ of $\mathbb R^3$ and express the
hypersurface $M$ by an immersion
$$
\begin{aligned}
\Phi:\mathbb R^3\supset\Omega& \longrightarrow \mathrm{Sol}_{m, n}^4 ,\\
(u,v,w)&\longmapsto(x(u,v,w),y(u,v,w),z(u,v,w),t(u,v,w)),
\end{aligned}
$$
such that
\begin{equation}\label{ppp:4.10}
\begin{aligned}
&\mathrm{d} \Phi(\partial_u)
=\left(x_u,y_u,z_u,t_u\right)=X_1,\\
&\mathrm{d} \Phi(\partial_v)
=\left(x_v,y_v,z_v,t_v\right)=X_2,\\
&\mathrm{d} \Phi(\partial_w)
=\left(x_w,y_w,z_w,t_w\right)=X_3.
\end{aligned}
\end{equation}
According to \eqref{ttt:2.1}, \eqref{ppp:4.5} and \eqref{ppp:4.9}, we also have
\begin{equation}\label{ppp:4.11}
X_1=(1,0,0,0),\ \ X_2=(0,de^{\beta t},0,\sqrt{1-d^2}),\ \  X_3=(0,0,1,0).
\end{equation}
The equations \eqref{ppp:4.10} and \eqref{ppp:4.11} show that $y$ and $t$ depend only on $v$, $x$ depends only on $u$, and $z$ depends only on $w$.

It follows from \eqref{ppp:4.10} and \eqref{ppp:4.11} that $x=u+c_5$, $z=w+c_6$, $t=\sqrt{1-d^2}v+c_7$, where $c_5$, $c_6$ and $c_7$ are constants. Thus, $v=\tfrac{t-c_7}{\sqrt{1-d^2}}$ and
\begin{equation}\label{ppp:4.12}
v_t=\tfrac{1}{\sqrt{1-d^2}}.
\end{equation}
In what follows we use $t$ instead of $v$ as a local coordinate.

Then, from  \eqref{ppp:4.10}, \eqref{ppp:4.11} and \eqref{ppp:4.12}, we derive that
\begin{equation}\label{oay:4.12}
y_t=y_vv_t=\tfrac{d}{\sqrt{1-d^2}}e^{\beta t}.
\end{equation}
To solve equation \eqref{oay:4.12}, we consider the following two subcases:

\vskip2mm
{\bf Case I-(1)-(i).} $\beta\neq0$ on $M$.

In this subcase, we have $m\neq n$. Solving the equation \eqref{oay:4.12},
we get $y=\tfrac{1}{\beta}\tfrac{d}{\sqrt{1-d^2}}e^{\beta t}+c_8$, where $c_8$ is a constant.

Using $x$ and $z$ as the new local coordinates, we obtain
$$
\Phi(x,t,z)=(x, \tfrac{1}{\beta}\tfrac{d}{\sqrt{1-d^2}}e^{\beta t}+c_8,z,t).
$$
After a left translation by $(0,-c_8, 0, 0)$, which is an isometry of $\mathrm{Sol}_{m,n(m\neq n)}^4$, we get
\begin{equation*}
\Phi(x,t,z)=(x,\tfrac{1}{\beta}\tfrac{d}{\sqrt{1-d^2}}e^{\beta t},z,t).
\end{equation*}
Since $0\leq d< 1$, we may assume that $d=\tanh(|\beta| r)$ for some constant $r\geq 0$. By taking the
reparametrization $x_1=x$, $x_2=z$, $x_3=t+\tfrac{1}{\beta}\ln(\cosh (|\beta| r))$,  we get
$$
\Phi(x_1,x_2,x_3)=(x_1,\tfrac{1}{\beta}e^{\beta x_3}\tanh (|\beta| r),x_2,x_3-\tfrac{1}{\beta}\ln(\cosh (|\beta| r))).
$$
This yields that, up to isometries of $\mathrm{Sol}_{m,n(m\neq n)}^4$, $M$ is an open part of $M_{3,r}$ for some $r\geq0$.

\vskip2mm
{\bf Case I-(1)-(ii).} $\beta=0$ on $M$.

It follows from $\beta=0$ that $m=n$ holds. Substituting $\beta=0$ into \eqref{oay:4.12}, we derive that $y_t=\tfrac{d}{\sqrt{1-d^2}}$, which implies that $y=\tfrac{d}{\sqrt{1-d^2}}t+c_9$, where $c_9$ is a constant.

Using $x$ and $z$ as the new local coordinates, we obtain $\Phi(x,t,z)=(x, \tfrac{d}{\sqrt{1-d^2}} t+c_9,z,t)$.
After a left translation by $(0,-c_9, 0, 0)$, which is an isometry of $\mathrm{Sol}_{m,m}^4$, we get
\begin{equation*}
\Phi(x,t,z)=(x,\tfrac{d}{\sqrt{1-d^2}} t,z,t).
\end{equation*}
By taking the reparametrization $x_1=x$, $x_2=t$, $x_3=z$,  we get
$$
\Phi(x_1,x_2,x_3)=(x_1,\tfrac{d}{\sqrt{1-d^2}} x_2,x_3,x_2).
$$
This shows that, up to isometries of $\mathrm{Sol}_{m,m}^4$, $M$ is an open part of $M_{4,d}$ for some fixed $0\leq d<1$.

\vskip2mm
{\bf Case I-(2).} $b=c=0$ on $M$.

In this subcase,  we have $N=aE_1+dE_4$, where $a^2+d^2=1$.
Up to the action of reflection in the $x$-coordinate,
we assume that $a=-\sqrt{1-d^2}$.

Consider the orthonormal frame field
\begin{equation}\label{ttt:4.5}
W_1=d E_1+\sqrt{1-d^2} E_4,\ \ W_2=E_2,\ \ W_3=E_3,
\end{equation}
by using \eqref{ttt:2.3}, \eqref{kkk:2.6}, \eqref{ttt:4.5}, and $d$ is a constant, we obtain
\begin{equation}\label{ttt:4.7}
\begin{aligned}
&\nabla_{W_1}W_1= \nabla_{W_1}W_2=\nabla_{W_1}W_3=\nabla_{W_2}W_3=\nabla_{W_3}W_2=0,\\
&\nabla_{W_2}W_1=-\sqrt{1-d^2}\beta\, W_2,\ \
\nabla_{W_2}W_2 =\sqrt{1-d^2} \beta\,  W_1,\\
&\nabla_{W_3}W_1=-\sqrt{1-d^2} \gamma\, W_3,\ \
\nabla_{W_3}W_3=\sqrt{1-d^2} \gamma W_1,
\end{aligned}
\end{equation}
\begin{equation}\label{zx:1.2}
AW_1 =d \alpha W_1,    \ \    AW_2= d \beta W_2,   \ \      AW_3 =d \gamma W_3.
\end{equation}
It follows from \eqref{ttt:4.7} that
\begin{equation}\label{ttt:4.8}
[W_1,W_2]=\sqrt{1-d^2} \beta W_2,\ \ [W_1,W_3]=\sqrt{1-d^2} \gamma W_3,\ \ [W_2,W_3]=0.
\end{equation}
Then it can be checked that all the Gauss and Codazzi equations are satisfied.

Applying \eqref{ttt:2.1} and \eqref{ttt:4.5}, we have
\begin{equation}\label{ttt:4.6}
W_1(t)=\sqrt{1-d^2},\ \ W_2(t)=W_3(t)=0.
\end{equation}

Now, we define a new frame field $\{X_i\}_{i=1}^3$ on $M$ as follows:
\begin{equation}\label{ttt:4.9}
X_1=W_1,\ \ X_2=e^{-\beta t}\, W_2,\ \   X_3=e^{-\gamma t}\,W_3.
\end{equation}
By applying \eqref{ttt:4.8}, \eqref{ttt:4.6} and \eqref{ttt:4.9}, we derive that $[X_i,X_j]=0$, $1\leq i<j\leq3$.

Then, we can locally identify $M$ with an open subset $\Omega$ of $\mathbb R^3$ and express the
hypersurface $M$ by an immersion
$$
\begin{aligned}
\Phi:\mathbb R^3\supset\Omega& \longrightarrow \mathrm{Sol}_{m,n}^4 ,\\
(u,v,w)&\longmapsto(x(u,v,w),y(u,v,w),z(u,v,w),t(u,v,w)),
\end{aligned}
$$
such that
\begin{equation}\label{ttt:4.10}
\begin{aligned}
&\mathrm{d} \Phi(\partial_u)
=\left(x_u,y_u,z_u,t_u\right)=X_1,\\
&\mathrm{d} \Phi(\partial_v)
=\left(x_v,y_v,z_v,t_v\right)=X_2,\\
&\mathrm{d} \Phi(\partial_w)
=\left(x_w,y_w,z_w,t_w\right)=X_3.
\end{aligned}
\end{equation}
According to \eqref{ttt:2.1}, \eqref{ttt:4.5} and \eqref{ttt:4.9}, we also have
\begin{equation}\label{ttt:4.11}
X_1=(d e^{\alpha t},0,0,\sqrt{1-d^2}),\ \ X_2=(0,1,0,0),\ \  X_3=(0,0,1,0).
\end{equation}
The equations \eqref{ttt:4.10} and \eqref{ttt:4.11} show that $x$ and $t$ depend only on $u$, $y$ depends only on $v$, and $z$ depends only on $w$.

It follows from \eqref{ttt:4.10} and \eqref{ttt:4.11} that $y=v+c_{10}$, $z=w+c_{11}$, where $c_{10}$ and $c_{11}$ are constants. Using \eqref{ttt:4.10} and \eqref{ttt:4.11}, we also get $t=\sqrt{1-d^2}u+c_{12}$,  where $c_{12}$ is a constant.
Thus, $u=\tfrac{t-c_{12}}{\sqrt{1-d^2}}$ and
\begin{equation}\label{ttt:4.12}
u_t=\tfrac{1}{\sqrt{1-d^2}}.
\end{equation}
In what follows we use $t$ instead of $u$ as a local coordinate.

Then, from  \eqref{ttt:4.10}, \eqref{ttt:4.11} and \eqref{ttt:4.12}, we derive that $x_t=x_uu_t=\tfrac{d}{\sqrt{1-d^2}}e^{\alpha t}$, which implies that $x=\tfrac{1}{\alpha}\tfrac{d}{\sqrt{1-d^2}}e^{\alpha t}+c_{13}$, where $c_{13}$ is a constant.

Using $y$ and $z$ as the new local coordinates, we obtain $\Phi(t,y,z)=(\tfrac{1}{\alpha}\tfrac{d}{\sqrt{1-d^2}}e^{\alpha t}+c_{13},y,z,t)$.
After a left translation by $(-c_{13}, 0, 0, 0)$, which is an isometry of $\mathrm{Sol}_{m,n}^4$, we get
\begin{equation*}
\Phi(t,y,z)=(\tfrac{1}{\alpha}\tfrac{d}{\sqrt{1-d^2}}e^{\alpha t},y,z,t).
\end{equation*}
Since $0\leq d< 1$, we may assume that $d=\tanh(-\alpha r)$ for some constant $r\geq 0$. By taking the
reparametrization $x_1=y$, $x_2=z$, $x_3=t+\tfrac{1}{\alpha}\ln(\cosh(-\alpha r))$, we get
$$
\Phi(x_1,x_2,x_3)=(\tfrac{1}{\alpha}e^{\alpha x_3}\tanh (-\alpha r),x_1,x_2,x_3-\tfrac{1}{\alpha}\ln(\cosh (-\alpha r))).
$$
This shows that, up to isometries of $\mathrm{Sol}_{m,n}^4$, $M$ is an open part of $M_{5,r}$ for some $r\geq0$.

\vskip2mm
{\bf Case I-(3).} $a=b=0$ on $M$.

In this subcase,  we have $N=cE_3+dE_4$, where $c^2+d^2=1$.
Up to the action of reflection in the $z$-coordinate,
we assume that $c=-\sqrt{1-d^2}$. Put
\begin{equation}\label{ooo:4.5}
W_1=E_1,\ \ W_2=E_2,\ \ W_3=dE_3+\sqrt{1-d^2} E_4.
\end{equation}
It follows from \eqref{ttt:2.3}, \eqref{kkk:2.6}, \eqref{ooo:4.5} and $d$ is a constant that
\begin{equation}\label{ooo:4.7}
\begin{aligned}
&\nabla_{W_1}W_1=\sqrt{1-d^2} \alpha W_3,\ \ \ \
\nabla_{W_1}W_3=-\sqrt{1-d^2} \alpha W_1,\\
&\nabla_{W_2}W_2=\sqrt{1-d^2} \beta W_3,\ \ \ \
\nabla_{W_2}W_3=-\sqrt{1-d^2} \beta W_2, \\
&\nabla_{W_1}W_2=\nabla_{W_2}W_1=\nabla_{W_3}W_1=\nabla_{W_3}W_2=\nabla_{W_3}W_3=0,
\end{aligned}
\end{equation}
\begin{equation}\label{zx:1.3}
AW_1 =d \alpha W_1, \ \   AW_2= d \beta W_2,   \ \   AW_3 =d \gamma W_3.
\end{equation}
From \eqref{ooo:4.7}, we have
\begin{equation}\label{ooo:4.8}
[W_1,W_2]=0,\ \ [W_1,W_3]=-\sqrt{1-d^2} \alpha W_1,\ \ [W_2,W_3]=-\sqrt{1-d^2} \beta W_2.
\end{equation}
Then it can be checked that all the Gauss and Codazzi equations are satisfied.

By using \eqref{ttt:2.1} and \eqref{ooo:4.5}, we get
\begin{equation}\label{ooo:4.6}
W_1(t)=W_2(t)=0,\ \ W_3(t)=\sqrt{1-d^2}.
\end{equation}

Now, we define a new frame field $\{X_i\}_{i=1}^3$ on $M$ as follows:
\begin{equation}\label{ooo:4.9}
X_1=e^{-\alpha t}W_1,\ \ X_2=e^{-\beta t}W_2,\ \  X_3=\,W_3.
\end{equation}
By applying \eqref{ooo:4.8}, \eqref{ooo:4.6} and \eqref{ooo:4.9}, we derive that $[X_i,X_j]=0$, $1\leq i<j\leq3$.

Then, we can locally identify $M$ with an open subset $\Omega$ of $\mathbb R^3$ and express the
hypersurface $M$ by an immersion
$$
\begin{aligned}
\Phi:\mathbb R^3\supset\Omega& \longrightarrow \mathrm{Sol}_{m,n}^4 ,\\
(u,v,w)&\longmapsto(x(u,v,w),y(u,v,w),z(u,v,w),t(u,v,w)),
\end{aligned}
$$
such that
\begin{equation}\label{ooo:4.10}
\begin{aligned}
&\mathrm{d} \Phi(\partial_u)
=\left(x_u,y_u,z_u,t_u\right)=X_1,\\
&\mathrm{d} \Phi(\partial_v)
=\left(x_v,y_v,z_v,t_v\right)=X_2,\\
&\mathrm{d} \Phi(\partial_w)
=\left(x_w,y_w,z_w,t_w\right)=X_3.
\end{aligned}
\end{equation}
According to \eqref{ttt:2.1}, \eqref{ooo:4.5} and \eqref{ooo:4.9}, we also have
\begin{equation}\label{ooo:4.11}
X_1=(1,0,0,0),\ \ X_2=(0,1,0,0),\ \ X_3=(0,0,d e^{\gamma t},\sqrt{1-d^2}).
\end{equation}
The equations \eqref{ooo:4.10} and \eqref{ooo:4.11} show that $z$ and $t$ depend only on $w$, $x$ depends only on $u$, and $y$ depends only on $v$.

It follows from \eqref{ooo:4.10} and \eqref{ooo:4.11} that $x=u+c_{14}$, $y=v+c_{15}$, where $c_{14}$ and $c_{15}$ are constants. Using \eqref{ooo:4.10} and \eqref{ooo:4.11}, we also get $t=\sqrt{1-d^2}w+c_{16}$,  where $c_{16}$ is a constant.
Thus, $w=\tfrac{t-c_{16}}{\sqrt{1-d^2}}$ and
\begin{equation}\label{ooo:4.12}
w_t=\tfrac{1}{\sqrt{1-d^2}}.
\end{equation}
In what follows we use $t$ instead of $w$ as a local coordinate.

Then, from  \eqref{ooo:4.10}, \eqref{ooo:4.11} and \eqref{ooo:4.12}, we derive that $z_t=z_ww_t=\tfrac{d}{\sqrt{1-d^2}}e^{\gamma t}$, which together with $\gamma>0$ implies that $z=\tfrac{1}{\gamma}\tfrac{d}{\sqrt{1-d^2}}e^{\gamma t}+c_{17}$, where $c_{17}$ is a constant.

Using $x$ and $y$ as the new local coordinates, we obtain
$\Phi(x,y,t)=(x,y,\tfrac{1}{\gamma}\tfrac{d}{\sqrt{1-d^2}}e^{\gamma t}+c_{17},t)$.
After a left translation by $(0,0,-c_{17}, 0)$, we get
\begin{equation*}
\Phi(x,y,t)=(x,y,\tfrac{1}{\gamma}\tfrac{d}{\sqrt{1-d^2}}e^{\gamma t},t).
\end{equation*}
Since $0\leq d< 1$, we may assume that $d=\tanh(\gamma r)$ for some constant $r\geq 0$.
By taking the
reparametrization $x_1=x$, $x_2=y$, $x_3=t+\tfrac{1}{\gamma}\ln(\cosh (\gamma r))$,  we get
$$
\Phi(x_1,x_2,x_3)=(x_1,x_2,\tfrac{1}{\gamma}e^{\gamma x_3}\tanh (\gamma r),x_3-\tfrac{1}{\gamma}\ln(\cosh (\gamma r))).
$$
This shows that, up to isometries of $\mathrm{Sol}_{m,n}^4$, $M$ is an open part of $M_{6,r}$ for some $r\geq0$.

\vskip2mm
{\bf Case II.} $d^2=1$ on $M$.

In this case, the unit normal vector field is $N=E_4$, and $TM=\text{span}\{E_1, E_2, E_3\}$, where $\{E_i\}_{i=1}^4$ is defined by \eqref{ttt:2.1}. It follows from \eqref{ttt:2.0}, \eqref{ttt:2.1} and \eqref{ttt:2.2} that $M$ is an open part of $\{(x_1,x_2,x_3,r) \in \mathrm{Sol}_{m,n}^4 \mid x_1,x_2,x_3 \in \mathbb{R}\}$ for some $r \in \mathbb{R}$, which is congruent to $M_{7,0}$.

\vskip1mm
Conversely, according to Propositions \ref{prop:3.4}--\ref{prop:3.6}, we know that the hypersurfaces $M_{3,r}$, $M_{5,r}$, $M_{6,r}$ for some $r\geq 0$, $M_{4,d}$ for some $0\leq d<1$ and $M_{7,0}$
have constant angle functions.

In conclusion, we have completed the proof of Theorem \ref{thm:1.3}.
\end{proof}

Now, we give the proof of Theorem \ref{thm:1.4}.

\noindent
{\bf Proof of Theorem \ref{thm:1.4}.}

Let $M$ be a homogeneous hypersurface of $\mathrm{Sol}_{m,n}^4$ with unit normal vector field $N$, then $M$ is an orbit of some closed subgroup $G \subset \mathrm{Sol}_{m, n}^4$. For a fixed point $p_0\in M$ and any $p\in M$, there exists an isometry $\phi \in \mathrm{Sol}_{m,n}^4$ such that $\phi(M)=M$ and $\phi(p_0)=p$. These implies that $N_p = \pm {\mathrm d}\phi_{p_0}(N_{p_0})$. It follows from $E_i$ are left invariant vector fields on $\mathrm{Sol}_{m,n}^4$ that ${\mathrm d}\phi_{p_0}(E_i|_{p_0})=E_i|_p$ for $i=1,2,3,4$. Similar to the discussions of the proof of Theorem \ref{thm:1.2}, we derive
that the angle functions $a$, $b$, $c$ and $d$ are constants on $M$.
Then by Theorem \ref{thm:1.3}, up to isometries of $ \mathrm{Sol}_{m,n}^4$, $M$ is either $M_{3,r}$, $M_{5,r}$, $M_{6,r}$ for some $ r\geq0$, or $M_{4,d}$ for some $0\leq d<1$, or $M_{7,0}$.

Conversely, from Propositions \ref{prop:3.4}--\ref{prop:3.6}, we know that all these hypersurfaces are homogeneous.

In conclusion, we have completed the  proof of Theorem \ref{thm:1.4}. \qed

\section{ Proof of Theorem \ref{thm:1.6}}\label{sect:6}
Similarly, in order to prove Theorem \ref{thm:1.6}, we present the following theorem.

\begin{theorem}\label{thm:1.5}
Let $M$ be a hypersurface of $\mathrm{Nil}^4$ with constant angle functions $a$, $b$, $c$ and $d$. Then, up to isometries of $\mathrm{Nil}^4$, one of the following two cases occurs:
\begin{enumerate}[\rm(1)]
  \item $M$ is an open part of $M_{8,d}$ for some $0\leq d<1$;
  \item $M$ is an open part of  $M_{9,0}$.
\end{enumerate}
\end{theorem}

\begin{proof}
Assume that $M$ is a hypersurface of $\mathrm{Nil}^4$ with the unit normal vector field
$N=aE_1 + bE_2 + cE_3 + dE_4$, where $\{E_i\}_{i=1}^4$ is defined by \eqref{hhh:2.1} and
$a$, $b$, $c$, $d$ are constants.
From $a^2+b^2+c^2+d^2=1$, we get that $|d|\leq1$. Moreover, up to changing the sign of the unit normal vector field, we assume that $d\geq0$. Then we can consider the following two cases:

\vskip2mm
{\bf Case I.} $0\leq d<1$ on $M$.

In this case, we choose the orthonormal frame field $\{T_i\}_{i=1}^3$ as defined by \eqref{kkk:2.5}.
It follows from \eqref{hhh:2.3}, the Weingarten formula, $a$, $b$, $c$ and $d$ are constants that
\begin{equation*}
\begin{aligned}
AT_1&=-\tilde{\nabla}_{T_1}N=-\tilde{\nabla}_{bE_1-aE_2+dE_3-cE_4}(aE_1+bE_2+cE_3+dE_4)\\
&=-ab\tilde{\nabla}_{E_1}E_1-b^2\tilde{\nabla}_{E_1}E_2
-bc\tilde{\nabla}_{E_1}E_3-bd\tilde{\nabla}_{E_1}E_4
+a^2\tilde{\nabla}_{E_2}E_1+ab\tilde{\nabla}_{E_2}E_2\\
& \ \ \
+ac\tilde{\nabla}_{E_2}E_3+ad\tilde{\nabla}_{E_2}E_4
-ad\tilde{\nabla}_{E_3}E_1-bd\tilde{\nabla}_{E_3}E_2
-cd\tilde{\nabla}_{E_3}E_3-d^2\tilde{\nabla}_{E_3}E_4\\
& \ \ \
+ac\tilde{\nabla}_{E_4}E_1+bc\tilde{\nabla}_{E_4}E_2
+c^2\tilde{\nabla}_{E_4}E_3+cd\tilde{\nabla}_{E_4}E_4\\
&=\tfrac{1}{2}((bc-ad)E_1+(c(c-a)+d(b+d))E_2-(bc+ad)E_3+(a(a+c)-b(b+d))E_4)\\
&=(b^2c-abd-a(c^2 + d^2))T_1+\tfrac{1}{2}(-b^3+2abc-b^2d+a^2(b+d)+b(c^2-d^2)-d(c^2 + d^2))T_2\\
& \ \ \
+\tfrac{1}{2}(-a^3-a^2c+c^3+c(b + d)^2+a(b^2-c^2+2bd-d^2))T_3.
\end{aligned}
\end{equation*}
Similar calculations give that
\begin{equation*}
\begin{aligned}
AT_2& = \tfrac{1}{2} \left( -b^3+b^2 d - 2 a c d + a^2 (d-b) + b (c^2 - d^2)- d (c^2 + d^2) \right)T_1\\
&\ \ \ \  +(-b^2 c + b c d + (a - c) d^2)T_2 \\
&\ \ \ \ +\tfrac{1}{2} \left( -b (b^2 + (a - c)^2) - (a^2 + b^2 + 2 a c - c^2) d + b d^2 - d^3 \right)T_3,\\
AT_3 &= \tfrac{1}{2} \left( a^3 - a^2 c + a (b^2 + c^2 - d^2) + c (c^2-b^2+ 2 b d + d^2) \right)T_1\\
&\ \ \ \ +\tfrac{1}{2} \left( -2 a c d + a^2 (b + d) + (b - d) (b^2 - c^2 + d^2) \right)T_2\\
&\ \ \ \  +(c d (d-b) + a (c^2 + b d))T_3.
\end{aligned}
\end{equation*}
Using the symmetry of the shape operator $A$, we derive that
\begin{align}
a^2b-b^2d+ac(b+d)&=0,\label{hhh:4.2}\\
b^2c-a(a^2+c^2-bd)&=0,\label{hhh:4.3}\\
-b^3+abc-a^2(b+d)&=0.\label{hhh:4.4}
\end{align}

We claim that $a=b=0$ on $M$. Indeed, if $ab\neq0$ holds, from \eqref{hhh:4.4}, we derive that $d=\tfrac{abc-b^3}{a^2}-b$. Substituting this into \eqref{hhh:4.2} with $b\neq0$ yields
\begin{equation}\label{hhh:777}
a^2c^2-2ab^2c+a^4+a^2b^2+b^4=0.
\end{equation}
Regarding \eqref{hhh:777} as a quadratic equation in the variable $c$, its discriminant $\Delta$ is given by:
$$
\Delta=(-2ab^2)^2 - 4 a^2 (a^4 + a^2b^2 + b^4)=-4a^4(a^2 + b^2)<0,
$$
which implies that there does not exist the real solution for $c$. Therefore, we have $ab=0$.
If $a\neq0$ and $b=0$, from \eqref{hhh:4.3}, we derive that $a=0$, which is a contradiction.
Thus, we get $a=0$. Substituting this into \eqref{hhh:4.4} yields $b=0$.
We have verified this assertion.

Now, we have $N=cE_3+dE_4$, where $c^2+d^2=1$.
Up to the action of $\psi_1$ which defined by \eqref{zx:1.4},
we can assume that $c=-\sqrt{1-d^2}$.

We choose the orthonormal frame field on $M$:
\begin{equation}\label{hhh:4.5}
W_1=E_1,\ \ W_2=E_2,\ \ W_3=dE_3+\sqrt{1-d^2}E_4.
\end{equation}
From \eqref{hhh:2.3}, \eqref{kkk:2.6}, \eqref{hhh:4.5} and $d$ is a constant, we obtain
\begin{equation}\label{hhh:4.7}
\begin{aligned}
\nabla_{W_1}W_1 &=0,      & \nabla_{W_1}W_2 &= \tfrac{\sqrt{1-d^2}}{2}W_3,
& \nabla_{W_1}W_3 &=-\tfrac{\sqrt{1-d^2}}{2}W_2,\\
\nabla_{W_2}W_1 &=\tfrac{\sqrt{1-d^2}}{2}W_3        ,       & \nabla_{W_2}W_2 &=0,
&\nabla_{W_2}W_3 &=-\tfrac{\sqrt{1-d^2}}{2}W_1, \\
\nabla_{W_3}W_1 &=-\tfrac{\sqrt{1-d^2}}{2}W_2,       & \nabla_{W_3}W_2&=\tfrac{\sqrt{1-d^2}}{2}W_1,
& \nabla_{W_3}W_3 &=0,
\end{aligned}
\end{equation}
\begin{equation}\label{zx:1.5}
AW_1 =\tfrac{d}{2}W_2,   \ \ AW_2= \tfrac{d}{2}W_1+\tfrac{1}{2}W_3, \ \ AW_3 = \tfrac{1}{2}W_2.
\end{equation}
By using \eqref{hhh:4.7}, we get
\begin{equation}\label{hhh:4.8}
[W_1,W_2]=0,\ \ [W_1,W_3]=0,\ \ [W_2,W_3]=-\sqrt{1-d^2}W_1.
\end{equation}
Consequently, a direct computation confirms that all Gauss and Codazzi equations are satisfied.

The equations \eqref{hhh:2.1} and \eqref{hhh:4.5} leads to
\begin{equation}\label{hhh:4.6}
W_1(t)=W_2(t)=0, \ \ W_3(t)=\sqrt{1-d^2}.
\end{equation}

Now, we define a new frame field $\{X_i\}_{i=1}^3$ on $M$ as follows:
\begin{equation}\label{hhh:4.9}
X_1=W_1,\ \ X_2=-t\, W_1+ W_2,\ \   X_3=W_3.
\end{equation}
It follows from \eqref{hhh:4.8}, \eqref{hhh:4.6} and \eqref{hhh:4.9} that $[X_i,X_j]=0$ holds, $1\leq i<j\leq3$.

Then, we can locally identify $M$ with an open subset $\Omega$ of $\mathbb R^3$ and express the
hypersurface $M$ by an immersion
$$
\begin{aligned}
\Phi:\mathbb R^3\supset\Omega& \longrightarrow \mathrm{Nil}^4 ,\\
(u,v,w)&\longmapsto(x(u,v,w),y(u,v,w),z(u,v,w),t(u,v,w)),
\end{aligned}
$$
such that
\begin{equation}\label{hhh:4.10}
\begin{aligned}
&\mathrm{d} \Phi(\partial_u)
=\left(x_u,y_u,z_u,t_u\right)=X_1,\\
&\mathrm{d} \Phi(\partial_v)
=\left(x_v,y_v,z_v,t_v\right)=X_2,\\
&\mathrm{d} \Phi(\partial_w)
=\left(x_w,y_w,z_w,t_w\right)=X_3.
\end{aligned}
\end{equation}
From \eqref{hhh:2.1}, \eqref{hhh:4.5} and \eqref{hhh:4.9}, we get
\begin{equation}\label{hhh:4.11}
X_1=(1,0,0,0),\ \ X_2=(0,1,0,0),\ \  X_3=(\tfrac{d}{2}t^2,dt,d,\sqrt{1-d^2}).
\end{equation}
The equations \eqref{hhh:4.10} and \eqref{hhh:4.11} yields that $z$ and $t$ depend only on $w$, $x$ depends only on $u$ and $w$, and $y$ depends only on $v$ and $w$.

It follows from \eqref{hhh:4.10} and \eqref{hhh:4.11} that $t_w=\sqrt{1-d^2}$, which implies that $t=\sqrt{1-d^2}w+c_{18}$, where $c_{18}$ is a constant. It follows that $w=\tfrac{t-c_{18}}{\sqrt{1-d^2}}$ and
\begin{equation}\label{hhh:4.12}
w_t=\tfrac{1}{\sqrt{1-d^2}}.
\end{equation}
In what follows we use $t$ instead of $w$ as a local coordinate.

Then, from  \eqref{hhh:4.10}, \eqref{hhh:4.11} and \eqref{hhh:4.12}, we obtain
$$
x_t=\tfrac{d}{2\sqrt{1-d^2}}t^2,\ \ y_t=\tfrac{d}{\sqrt{1-d^2}}t,\ \ z_t=\tfrac{d}{\sqrt{1-d^2}},\ \
x_u=1,\ \ y_v=1.
$$
Solving these equations, we get $x(u,t)=u+\tfrac{d}{6\sqrt{1-d^2}}t^3+c_{19}$, $y(v,t)=v+\tfrac{d}{2\sqrt{1-d^2}}t^2+c_{20}$, $z=\tfrac{d}{\sqrt{1-d^2}}t+c_{21}$, where $c_{19}$, $c_{20}$ and $c_{21}$ are constants.
Hence, we obtain $\Phi(u,v,t)=(u+\tfrac{d}{6\sqrt{1-d^2}}t^3+c_{19},v+\tfrac{d}{2\sqrt{1-d^2}}t^2+c_{20},
\tfrac{d}{\sqrt{1-d^2}}t+c_{21},t)$.
After a left translation by $(-c_{19}, -c_{20}, -c_{21}, 0)$, which is an isometry of $\mathrm{Nil}^4$, we get
\begin{equation*}
\Phi(u,v,t)=(u+\tfrac{d}{6\sqrt{1-d^2}}t^3,v+\tfrac{d}{2\sqrt{1-d^2}}t^2,\tfrac{d}{\sqrt{1-d^2}}t,t).
\end{equation*}
This shows that $\sqrt{1-d^2}z-dt=0$ holds. Then from item (1) of Proposition \ref{prop:3.11}, we know that $M$ is an open part of $M_{8,d}$ for some $0\leq d<1$.

\vskip2mm
{\bf Case II.} $d^2=1$ on $M$.

In this case, the unit normal vector field is $N=E_4$, and $TM=\text{span}\{E_1, E_2, E_3\}$, where $\{E_i\}_{i=1}^4$ is defined by \eqref{hhh:2.1}. According to \eqref{hhh:2.0}, \eqref{hhh:2.1} and
\eqref{hhh:2.2}, we derive that $M$ is an open part of $\{(x_1,x_2,x_3,r) \in \mathrm{Nil}^4 \mid x_1,x_2,x_3 \in \mathbb{R}\}$ for some $r \in \mathbb{R}$, which is congruent to $M_{9,0}$.

Conversely, by Propositions \ref{prop:3.11} and \ref{prop:3.12}, we know that the hypersurfaces $M_{8,d}$ and $M_{9,0}$ have constant angle functions.

In conclusion, we have completed the proof of Theorem \ref{thm:1.5}.
\end{proof}

Now, we give the proof of Theorem \ref{thm:1.6}.

\noindent
{\bf Proof of Theorem \ref{thm:1.6}.}

Let $M$ be a homogeneous hypersurface of $\mathrm{Nil}^4$ with unit normal vector field $N$, then $M$ is an orbit of some closed subgroup $G \subset \mathrm{Nil}^4$. For a fixed point $p_0\in M$ and any $p\in M$. There exists an isometry $\phi \in \mathrm{Nil}^4$ such that $\phi(M)=M$ and $\phi(p_0)=p$,
which shows that $N_p = \pm {\mathrm d}\phi_{p_0}(N_{p_0})$. Then, similar to the proof of Theorem \ref{thm:1.2}, we obtain
that the angle functions $a$, $b$, $c$ and $d$ are constants on $M$.
Therefore according to Theorem \ref{thm:1.5}, up to isometries of $\mathrm{Nil}^4$, we know that $M$ is either $M_{8,d}$ for some $0\leq d<1$ or $M_{9,0}$.

Conversely, from Propositions \ref{prop:3.11} and \ref{prop:3.12}, we know that both $M_{8,d}$ for some $0\leq d<1$ and $M_{9,0}$ are homogeneous hypersurfaces.

In conclusion, we have completed the  proof of Theorem \ref{thm:1.6}. \qed






\begin{thebibliography}{11}

\bibitem{BM}
Belkhelfa, M., Mokni, H.:
{\it Classification of hypersurfaces in the four dimensional Thurston geometry $\mathrm{Sol}_{m, n}^4$}.
J. Geom. {\bf 116}(2), Art. 26, 12 pp. (2025)

\bibitem{BS3}
Berndt J., Suh Y.J. :
{\it Real hypersurfaces in Hermitian symmetric spaces},
Advances in Analysis and Geometry, vol. 5, De Gruyter, Berlin (2022)

\bibitem{BT}
Berndt, J., Tamaru, H.:
{\it Cohomogeneity one actions on noncompact symmetric spaces of rank one}.
Trans. Amer. Math. Soc. {\bf359}, 3425--3438 (2007)

\bibitem{CE}
Cartan, \'E.:
{\it Families de surfaces isoparam\'etriques dans les espaces \`a courbure constante}.
Ann. Mat. Pura Appl. {\bf17}(1), 177--191 (1938)

\bibitem{CR}
Cecil T.E., Ryan P.J.:
{\it Geometry of hypersurfaces},
Springer Monographs in Mathematics, Springer, New York (2015)

\bibitem{CQ}
Chi, Q-S.:
{\it The isoparametric story, a heritage of \'Elie Cartan}.
Proceedings of the International Consortium of Chinese Mathematicians 2018,
Int. Press, Boston, MA, 197--260 (2020)

\bibitem{DP}
de Lima, R.F., Pipoli, G.:
{\it Isoparametric hypersurfaces of $\mathbb{H}^{n}\times\mathbb{R}$ and $\mathbb{S}^{n}\times\mathbb{R}$}.
To appear in Ann. Sc. Norm. Super. Pisa Cl. Sci.
arXiv:2411.11506v2

\bibitem{DP2}
de Lima, R.F., Pipoli, G.:
{\it Isoparametric hypersurfaces in products of simply connected space forms}.
arXiv:2511.12527v1

\bibitem{DM1}
D'haene, M.:
{\it Thurston geometries in dimension four from a Riemannian perspective}.
arXiv:2401.05977v1

\bibitem{DWYZ}
D'haene, M.,  Wei, G.X., Yao, Z.K., Zhang, X.:
{\it Homogeneous hypersurfaces of the four-dimensional Thurston geometry $\mathrm{Sol}_{0}^{4}$}. arXiv:2508.10545v2

\bibitem{DDO1}
D\'iaz-Ramos, J. C., Dom\'inguez-V\'azquez, M., Otero, T.:
{\it Cohomogeneity one actions on symmetric spaces of noncompact type and higher rank}.
Adv. Math. {\bf428}, Art. 109165, 33 pp. (2023)

\bibitem{DDO}
D\'iaz-Ramos, J. C., Dom\'inguez-V\'azquez, M., Otero, T.:
{\it Homogeneous hypersurfaces in symmetric spaces}.
New trends in geometric analysis-Spanish Network of Geometric Analysis 2007--2021,
Springer, Cham, 141--190 (2023)

\bibitem{DHCB}
Djellali, N., Hasni, A., Cherif, A.M.,  Belkhelfa, M.:
{\it Classification of Codazzi and note on minimal hypersurfaces in $\mathrm{Nil}^4$}.
Int. Electron. J. Geom. {\bf16}(2), 707--714 (2023)

\bibitem{DFO}
Dom\'inguez-V\'azquez, M., Ferreira, T.A., Otero, T.:
{\it Polar actions on homogeneous $3$-spaces}.
Ann. Mat. Pura Appl. {\bf205}(2), 903--927 (2026)

\bibitem{DM}
Dom\'inguez-V\'azquez, M., Manzano, J.M.:
{\it Isoparametric surfaces in $\mathbb{E}(\kappa,\tau)$-spaces}.
Ann. Sc. Norm. Super. Pisa Cl. Sci. {\bf22}(1), 269--285 (2021)

\bibitem{EI}
Erjavec, Z., Inoguchi, J.:
{\it Minimal submanifolds in $\mathrm{Sol}_1^4$}.
Rev. R. Acad. Cienc. Exactas F\'is. Nat. Ser. A Mat. RACSAM {\bf117}(4), Art. 156, 36 pp. (2023)

\bibitem{EI1}
Erjavec, Z., Inoguchi, J.:
{\it Codazzi and totally umbilical hypersurfaces in $\mathrm{Sol}_1^4$}.
Glasg. Math. J. {\bf67}(3), 487--493 (2025)

\bibitem{FR}
Filipkiewicz, R.:
{\it Four dimensional geometries}.
PhD thesis, University of Warwick (1983)

\bibitem{GMY1}
Gao, D., Ma, H., Yao, Z.K.:
{\it Isoparametric hypersurfaces in product spaces of space forms}.
Differential Geom. Appl. {\bf95}, Art. 102155, 8 pp. (2024)

\bibitem{GMY}
Gao, D., Ma, H., Yao, Z.K.:
{\it On hypersurfaces of $\mathbb{H}^{2}\times\mathbb{H}^{2}$}.
Sci. China Math. {\bf67}(2), 339--366 (2024)

\bibitem{GQTY}
Ge, J.Q., Qian, C., Tang, Z.Z., Yan, W.J.:
{\it An overview of the development of isoparametric theory}.
Sci. Sin. Math. {\bf55}(1), 145--168 (2025)

\bibitem{HL}
Harvey, R., Lawson, H.B. Jr.:
{\it Calibrated geometries}.
Acta Math. {\bf148}, 47--157 (1982)

\bibitem{HL1}
Hsiang, W-Y., Lawson, H.B. Jr.:
{\it Minimal submanifolds of low cohomogeneity}.
J. Differential Geom. {\bf5}, 1--38 (1971)

\bibitem{KA}
Kollross, A.:
{\it A classification of hyperpolar and cohomogeneity one actions}.
Trans. Amer. Math. Soc. {\bf354}(2), 571--612 (2002)

\bibitem{MP}
		Meeks W.H. and P\'{e}rez J.:
{\it Constant mean curvature surfaces in metric Lie groups}.
Geometric Analysis: Partial Differential Equations and Surfaces,
Contemporary Mathematics (AMS) vol. 570, 25--110 (2012)

\bibitem{SS}
Sanmart\'in-L\'opez, V., Solonenko, I.:
{\it Classification of cohomogeneity-one actions on symmetric spaces of noncompact type}.
arXiv:2501.05553v2

\bibitem{SB}
Segre, B.:
{\it Famiglie di ipersuperficie isoparametriche negli spazi euclidei ad un qualunque numero di dimensioni}.
Atti Accad. Naz. Lincei Rend. Cl. Sci. Fis. Mat. Natur. {\bf27}, 203--207 (1938)

\bibitem{TR}
Takagi, R.:
{\it On homogeneous real hypersurfaces in a complex projective space}.
Osaka Math. J. {\bf10}, 495--506 (1973)

\bibitem{TT}
Takagi, R., Takahashi, T.:
{\it On the principal curvatures of homogeneous hypersurfaces in a sphere}.
Differential geometry (in honor of Kentaro Yano),
Kinokuniya Book Store, Tokyo, 469--481 (1972)

\bibitem{TXY1}
Tan, H.X., Xie Y.Q., Yan, W.J.:
{\it Isoparametric hypersurfaces in $\mathbb{S}^n\times\mathbb{R}^m$ and $\mathbb{H}^n\times\mathbb{R}^m$}.
To appear in Sci. China Math. arXiv:2511.07782v2

\bibitem{TW}
Thurston, W.P.:
{\it Three-dimensional geometry and topology}.
Princeton Mathematical Series, Princeton University Press, Princeton, NJ, x+311 pp. (1997)

\bibitem{UF}
Urbano, F.:
{\it On hypersurfaces of $\mathbb{S}^{2}\times\mathbb{S}^{2}$}.
Comm. Anal. Geom. {\bf27}(6), 1381--1416 (2019)

\bibitem{WCTC}
Wall, C.T.C.:
{\it Geometric structures on compact complex analytic surfaces}.
Topology {\bf25}(2), 119--153 (1986)

\end{thebibliography}
\end{document}